\documentclass[12pt]{article}
\usepackage{latexsym, amsmath, amsfonts, amsthm, amssymb}
\usepackage{times}
\usepackage{a4wide}
\usepackage{times}
\usepackage{graphicx, tocloft}
\usepackage{mathrsfs}
\usepackage{multicol}
\usepackage[shortlabels]{enumitem}
\usepackage{url}
\usepackage{hyperref}
\usepackage{thm-restate}

\def \RR {\mathbb R}
\def \NN {\mathbb N}

\def \ZZ {\mathbb Z}

\def \eps {\varepsilon}

\def \tD {\tilde D}

\newtheorem{theorem}{Theorem}[section]
\newtheorem{theorem'}{Theorem}
\newtheorem{definition}[theorem]{Definition}

\newtheorem{lemma}[theorem]{Lemma}

\newtheorem{proposition}[theorem]{Proposition}
\newtheorem{corollary}[theorem]{Corollary}
\newtheorem*{claim*}{Claim}

\theoremstyle{definition}
\newtheorem{remark}[theorem]{Remark}

\def\myffrac#1#2 in #3{\raise 2.6pt\hbox{$#3 #1$}\mkern-1.5mu\raise 0.8pt\hbox{$
		#3/$}\mkern-1.1mu\lower 1.5pt\hbox{$#3 #2$}}

\def\qed{\hfill $\vcenter{\hrule height .3mm
		\hbox {\vrule width .3mm height 2.1mm \kern 2mm \vrule width .3mm
			height 2.1mm} \hrule height .3mm}$ \bigskip}

\def \Len {{\rm Len}}
\def \diam {{\rm diam}}
\newcommand{\norm}[1]{\left\lVert#1\right\rVert}
\newcommand{\tr}{\triangle}

\AtEndDocument{%
  \par
  \medskip
  \begin{tabular}{@{}l@{}}%
    \textsc{Department of Mathematics, Weizmann Institute of Science, Rehovot 76100 Israel}\\
    \textit{E-mail address}: \texttt{rotem.assouline@weizmann.ac.il}
  \end{tabular}}

\title{A finiteness principle for distance functions on Riemannian surfaces with H{\"o}lder continuous curvature}
\author{Rotem Assouline}
\date{}

\begin{document}

\maketitle
\abstract{We study distance functions from geodesics to points on Riemannian surfaces with H{\"o}lder continuous Gauss curvature, and prove a finiteness principle in the spirit of Whitney extension theory for such functions. Our result can be viewed as a finiteness principle for isometric embedding of a certain type of metric spaces into Riemannian surfaces, with control over the H{\"o}lder seminorm of the Gauss curvature.}

\section{Introduction}

Consider a function $\rho$ measuring the distance from a unit-speed geodesic $\gamma$ on a Riemannian surface $(M,g)$ to a point $p \in M$. Certain analytic properties of $\rho$ are related to the Gauss curvature $K$ of the metric $g$. For instance, if $K$ is nonpositive, then $\rho$ is convex; and if $K$ is $\alpha$-H{\"o}lder, which will be the case throughout this paper, then $\rho$ is $C^{2,\alpha}$, assuming that $\gamma$ is contained in a punctured strongly convex neighbourhood of $p$, see e.g. Lemma \ref{geoeqlemma} below. 

\medskip
Given a 1-Lipschitz, positive function $\rho$ on an interval $I \subseteq \RR$, is it possible to tell whether $\rho$ arises as a distance function from a unit-speed geodesic $\gamma : I \to M$ to a point $p \in M$, for an ambient surface $M$ having bounded geometry in some sense?

\medskip
For example, consider the function $\rho(t) = \sqrt{1 + t^2} - c$ for $0 \le c < 1$. If $c = 0$ then $\rho$ is the distance function from the line $\gamma(t) = (1,t)$ to the point $p = (0,0)$ on the Euclidean plane. But if we take $c$ which is very close to $1$ and try to construct a Riemannian surface admitting $\rho$ as a distance function from a unit-speed geodesic to a point, we will find that the Gauss curvature of the surface will have to be very large somewhere. This follows for example from Lemma \ref{kappaivt} below.

\medskip
As a slightly more delicate example one can take $\rho(t) = \sqrt{\eps^2 + t^2} + \eps^{3 + \beta}$ for $0 < \beta \le 1$ and $0 < \eps \ll 1$. The function $\rho$ is very close to a Euclidean distance function. However, if we wanted to realize $\rho$ as a distance function on a Riemannian surface $M$, then the $\alpha$-H{\"o}lder seminorm of the Gauss curvature of $M$ would have to be of order at least $\eps^{\beta - \alpha}$, for every $0 < \alpha \le 1$ (this too can be shown using Lemma \ref{kappaivt}). In particular, the $\beta$-H{\"o}lder seminorm of $K$ must be bounded below by a constant independent of $\eps$.

\medskip 
By a {\it$C^{2,\alpha}$ Riemannian surface} we mean a two-dimensional Riemannian manifold with bounded, $\alpha$-H{\"o}lder continuous Gauss curvature, and with injectivity radius bounded from below. We use the same parameter $H$ to bound all three quantities, see Definition \ref{C2asurfacedef} below for the precise definition of a $C^{2,\alpha}$ Riemannian surface with constant $H$. 

\medskip
1-Lipschitz continuity and $C^{2,\alpha}$-smoothness are properties which can be detected by looking at 2 and 4-point samples from the function $\rho$, respectively. Does the property of \emph{being a distance function on a $C^{2,\alpha}$ surface}, together with a quantitative bound on the constant $H$, also boil down to some condition on the values of $\rho$ on $k$-point subsets, for some universal $k$?

\medskip
Our main result is the following finiteness principle, which answers this question to the affirmative assuming that $\rho$ is bounded by a small enough constant.

\begin{theorem'}\label{mainthm}
	There exist universal constants $C_1, C_2$ such that the following holds:

Let $\rho : I \to (0,C_1]$, where $I \subseteq \RR$ is an open bounded interval. Assume that for every subset $S \subseteq I$ consisting of at most 12 points, there exists a $C^{2,\alpha}$ Riemannian surface $(M_S, g_S)$ with constant 1, a point $p_S \in M_S$ and a unit speed geodesic $\gamma_S: I \to M_S$ such that $d_{g_S}(\gamma_S(s), p) = \rho(s)$ for every $s \in S$.

	Then, there exists a $C^{2,\alpha}$ Riemannian surface $M$ with constant $C_2$, a point $p \in M$ and a unit - speed geodesic $\gamma : I \to M$ such that $d_g(\gamma(t), p) = \rho(t)$ for every $t \in I$.
\end{theorem'}

\begin{remark}\normalfont
The condition on the twelve points in Theorem \ref{mainthm} is  essentially given in Proposition \ref{rhoestppn}. The fact that twelve points suffice can be roughly understood in the following way: using three nearby points, one can evaluate the second derivative of $\rho$ at a given point; two such ``three-point clusters'' are required in order to estimate the radial derivative of a quantity depending on $\rho,\dot\rho$ and $\ddot\rho$; and twelve points are required in order to evaluate the difference between the value of this radial derivative at two different scales. Thus verifying the conditions in Proposition \ref{rhoestppn} requires arranging the twelve points in a configuration involving three different scales.

We do not know whether the number 12 in Theorem \ref{mainthm} is sharp. Since every four-point metric space containing one collinear triple always embeds in a surface of constant curvature (see below), the sharp number is at least 4.
\end{remark}

Theorem \ref{mainthm} is inspired by finiteness principles from the theory of extensions of functions, initiated by Whitney \cite{W34}. For a comprehensive survey of Whitney extension theory, including finiteness principles, we refer the reader to Fefferman and Israel \cite{Fe20}. 

\medskip
Theorem \ref{mainthm} can be viewed as a special case of the following problem: is there a finiteness principle for isometric embeddings of metric spaces into Riemannian manifolds? i.e., can embeddability of a given metric space into a sufficiently regular Riemannian manifold be detected by examining its finite subspaces of cardinality bounded by some constant depending on the dimension and the regularity class alone? While it is hard to believe that the answer is generally yes, it might be the case that some classes of metric spaces do admit such a finiteness principle. 

\medskip The problem of embedding a metric space into Euclidean space goes back to Cayley \cite{Ca}. Menger \cite{Me} and Blumenthal \cite{Blum} considered isometric embeddings into spaces of constant curvature, and proved that a metric space $X$ embeds isometrically into the $n$-dimensional space of constant curvature $K_0 \in \RR$ if and only if every finite subspace of $X$ constisting of at most $n+3$ points admits such an embedding. Wald \cite{Wa} studied isometric embeddings of four-point metric spaces into surfaces of constant curvature, and proved that if $X$ is a metric space consisting of four points, of which exactly three points are collinear (i.e. embed isometrically in $\RR$), then there exists a unique $K_0 \in \RR$ such that $X$ embeds isometrically into the surface of constant curvature $K_0$. Berestovskii \cite{Be} characterized length spaces with two-sided curvature bounds in the sense of Alexandrov (see e.g. \cite{BI}) in terms of the embedding curvatures of their four-point subspaces.  Approximate interpolation of metric spaces by Riemannian manifolds of bounded geometry was studied in \cite{FIKLN}.

\medskip In order to construct the surface $M$ from the conclusion of Theorem \ref{mainthm}, we use the following characterization of the coefficients of $C^{2,\alpha}$ metrics (see Definition \ref{C2alphametricdef}) in polar normal coordinates, which we prove in Section \ref{coeffsec}. 

\begin{restatable}{theorem'}{Gthm}\label{Gthm}
    Let $0<\alpha\le 1$, let $R>0$ and let $G$ be a positive function defined on $B_R(0)\setminus\{0\}\subseteq \RR^2$. The two statements below are equivalent in the following sense: if 1. holds for some $\tilde H<\pi^2/4R^2$, then 2. holds with $ H=C\tilde H$, and if 2. holds for some $H< \pi^2/64R^2$, then 1. holds with $\tilde H=CH$, where $C>0$ is a universal constant.
    \begin{enumerate}
        \item The metric
        \begin{equation}\label{metricinpolarconst}
            g=dr^2+G^2d\theta^2.
        \end{equation}
		on $B_R(0)$ is a strongly convex $C^{2,\alpha}$ metric with constant at most $\tilde H$.
        \item The following conditions hold:
        \begin{enumerate}
            \item The function $G$ is $C^2$ in $r$.
            \item $G\to 0$ and $G/r\to 1$ as $r\to 0$.
            \item The function $K:=-\partial_r^2G/G$ satisfies $\norm{K}_\infty\le  H$ and $\norm{K}_\alpha\le  H^{1+\alpha/2}$.
         \end{enumerate}
    \end{enumerate}
\end{restatable}

Note that beyond the bound on the H{\"o}lder seminorm of $K$, no regularity in the variable $\theta$ is required. This agrees with the observation of Hartman \cite{hartman} that metrics with $C^{2,\alpha}$ coefficients in some coordinate system may fail to be even differentible in polar coordinates, except in the radial direction.

\medskip \emph{Outline of the proof of Theorem \ref{mainthm}:} Let $\rho$ satisfy the hypothesis of Theorem \ref{mainthm}. We construct a metric $g$ on a disc $D \subseteq \RR^2$, which in polar normal coordinates takes the form $g = dr^2 + G^2 d\theta^2$. Theorem \ref{Gthm} provides us with sufficient conditions for such a metric to be $C^{2,\alpha}$ with constant $O(H)$. We also need to construct a curve $\gamma : I \to D$ of the form $\gamma = \rho e^{i\phi}$ for some function $\phi : I \to [-\pi,\pi)$. For $\gamma$ taking this form, the relation $d_g(\gamma(t),0) = \rho(t)$ automatically holds, as we are in polar normal coordinates and within the radius of strong convexity. The functions $G$ and $\phi$ now need to be chosen such that $\gamma$ is a unit-speed geodesic with respect to $g$. To this end we make some preliminary ``guess" $\phi_0$ for the function $\phi$, and define first  $\gamma_0 = \rho e^{i\phi_0}$. Using the geodesic equation, we determine the function $h : = \partial_rG/G$ on the curve $\gamma_0$. We then extend $h$ to the entire disc $D$ (more precisely, we extend the function $f : = h - \cot_{K_0}r$ for some $K_0\in\RR$, see \S\ref{Riccatifacts} for the definition of $\cot_{K}$). This is done in Proposition \ref{fprop}. In order to prove that there exists such an extension satisfying the estimates required by Theorem \ref{Gthm}, we use the hypothesis of Theorem \ref{mainthm} together with results on distance functions on $C^{2,\alpha}$ surfaces which we prove in Section \ref{distfuncsec}. After the function $h$ has been extended, we determine the function $\phi$, deform $\gamma_0$ into the final curve $\gamma$ and obtain $G$ by integration.

\medskip The rest of the paper is organized as follows. In Section \ref{preliminaries} we mention relevant facts about polar normal coordinates on Riemannian surfaces, prove a stability lemma for the Riccati equation (Lemma \ref{riccatilemma}) which will serve as a technical tool in subsequent sections, and state a version of Whitney's extension theroem which will be used in the proof of Theorem \ref{mainthm}. In section \ref{coeffsec} we prove Theorem \ref{Gthm}. In Section \ref{distfuncsec} we study functions of the form $\rho : = d_g(\gamma(\cdot),p)$, where $g$ is a $C^{2,\alpha}$ Riemannian metric, $p$ is a point on $M$ and $\gamma$ is a unit-speed geodesic, and in some sense characterize functions arising in this way. In section \ref{mainthmproofsec} we prove Theorem \ref{mainthm}.

\medskip
\textbf{Acknowledgements} The author would like to express his deep gratitude to his advisor at the Weizmann Institute of Science, prof. Bo'az Klartag, for close guidance and support. Partially supported by the Israel Science Foundation (ISF).

\section{Preliminaries}\label{preliminaries}

\medskip For two quantities $A,B$, we write $A\lesssim B$ or $A=O(B)$ if there is a positive constant $C$ such that $A\le CB$, and $A\sim B$ if there is a positive constant $C$ such that $C^{-1}B\le A\le CB$. Unless stated otherwise, the constant $C$ is universal.

\medskip For a point $p$ in a metric space, and a positive number $R$, the open ball of radius $R$ centered at $p$ is denoted by $B_R(p)$.

\medskip A \emph{curve} on a manifold is a path which has a well defined continuous tangent.

\medskip Let $0 < \alpha \le 1$. A function $f$ on a metric space $(X,d)$ is \emph{$\alpha$ - H{\"o}lder} if there exists a constant $C$ such that 
$$|f(x) - f(y)| \le C \cdot d(x,y) ^ \alpha,$$
and the minimal $C$ with this property is called the H{\"o}lder seminorm of $f$ and will be denoted by $\norm{f}_\alpha$. We record the following trivial, yet useful, properties of the H{\"o}lder seminorm:
\begin{equation}\begin{split}\label{holderprops}\norm{fg}_\alpha \le \norm{f}_\alpha\norm{g}_\infty + \norm{f}_\infty\norm{g}_\alpha, &\qquad \norm{f\circ g}_{\alpha\beta} \le \norm{f}_\alpha\norm{g}_{\beta}^{\alpha}, \quad 0<\alpha,\beta\le 1,\\  \norm{f}_\alpha \le \diam(X)^{\beta-\alpha}\norm{f}_\beta \qquad & 0<\alpha\le\beta\le1. \end{split}\end{equation}
Here $\diam X = \sup\{d(x,y) \mid x,y\in X\}$ is the diameter of $X$, and
$$\norm{f}_\infty = \sup_{x\in X}|f(x)|$$

(all the functions in this paper are continuous).

\medskip For a function $f: \RR \to \RR$ we write $f \in C^{k,\alpha}(\RR)$ if $f\in C^k(\RR)$ and $f^{(k)}$ is $\alpha$-H{\"o}lder.

\medskip We use both the notation $u'$ and $\dot u$ for the derivative of a function $u$ of one variable.

\subsection{$C^{2,\alpha}$ Riemannian surfaces}

\medskip A Riemannian manifold $(M,g)$ is said to be $C^2$ if each point of $M$ is contained in the domain of a coordinate chart in which the coefficients of the metric $g$ are $C^2$ functions (the manifold $M$ is assumed to posses a smooth structure).

\begin{definition}\label{C2asurfacedef}\normalfont
    Let $H>0$ and $0<\alpha\le 1$. A \emph{$C^{2,\alpha}$ Riemannian surface with constant $H$} is a complete two-dimensional $C^2$ Riemannian manifold satisfying 
    \begin{align*}
        \norm{K}_\infty&\le H,& \norm{K}_\alpha&\le H^{1+\alpha/2} & &\text{ and }&  \mathrm{inj}(M)&\ge \pi/\sqrt{H},
    \end{align*}
    where $K$ is the (Gaussian) curvature of $(M,g)$, and its H{\"o}lder seminorm is taken with respect to the Riemannian metric.
\end{definition}  

\medskip We say the $(M,g)$ is a $C^{2,\alpha}$ Riemannian surface (or simply a $C^{2,\alpha}$ surface) if there exists some $H>0$ such that the requirements of Definition \ref{C2asurfacedef} are satisfied.

\begin{remark}\label{scalermk}\normalfont
        If a surface $(M,\lambda^2g)$ is obtained from a surface $(M,g)$ by multiplication of the metric $g$ by a scalar $\lambda^2>0$, and $g$ is a $C^{2,\alpha}$ surface with constant $H$, then $(M,\lambda^2g)$ is $C^{2,\alpha}$ with constant $\lambda^{-2}H$.
\end{remark}

\medskip Examples of $C^{2,\alpha}$ Riemannian surfaces include simply-connected, complete surfaces of constant curvature, graphs of $C^{2,\alpha}$ functions and level sets of of $C^{2,\alpha}$ functions from $\RR^3$ to $\RR$ whose gradient is bounded away from zero, and metrics obtained from the Euclidean metric on $\RR^2$ by introducing a $C^{2,\alpha}$ conformal factor which is bounded away from zero.

\medskip The coefficients $g_{ij}$ of the metric of a $C^{2,\alpha}$ surface are $C^{2,\alpha}$ functions in some coordinate system:

\begin{theorem}[Kazdan and Deturck \cite{KD}] \label{Kazdet}
        Let $(M,g)$ be a $C^{2,\alpha}$ Riemannian surface. Then any point in $M$ is contained in the domain of a coordinate chart in which the components of $g$ are $C^{2,\alpha}$ functions.
    \end{theorem}

\begin{definition}\label{C2alphametricdef}\normalfont
A Riemannian metric on an open disk $D \subseteq \RR^2$ is said to be a \emph{$C^{2,\alpha}$ metric with constant $H$} if it satisfies the coditions of Definition \ref{C2asurfacedef}, excluding completeness and bounded injectivity radius.
\end{definition}

\subsection{Polar coordinates, the Riccati equation and curvature comparison}\label{Riccatifacts}

\medskip About any point $p$ in a Riemannian surface $(M,g)$ one can introduce normal coordinates, in which the coordinate map is the exponential map $\exp_p$, and lines through the origin correspond to geodesics emanating from $p$.
Using polar coordinates in $\RR^2$, one obtains a set of coordinates $(r,\theta)$, called \textit{polar normal coordinates}. As a consequence of the Gauss lemma, the metric $g$ can be written in polar normal coordinates as
\begin{equation}\label{metricinpolar}
    g=dr^2+G^2d\theta^2
\end{equation}
for some function $G=G(r,\theta)$.

\medskip When working in polar normal coordinates, we shall use the complex notation $re^{i\theta}$ to denote a point whose coordinates are $(r,\theta)$. The coordinate $\theta$ will vary between $-\pi$ and $\pi$.

A \textit{normal disc} on a Riemannian surface is a Riemannian disc on which normal coordinates may be defined (i.e. a disk $B_R(p)$ such that $\exp_p:B_R(0)\to B_R(p)$ is a diffeomorphism, where $B_R(0)\subseteq T_pM$). When $M$ is $C^{2,\alpha}$ with constant $H$, then by definition, any point is the center of a normal disc of radius $\pi/\sqrt{H}$. Whenever polar normal coordinates are used, the domain will be such a disc.

\medskip Polar normal coordinates are not optimal in terms of regularity. Indeed, by Theorem \ref{Kazdet}, if $M$ is a $C^{2,\alpha}$ surface, then there exists a coordinate chart about $p$ in which the coefficients of $g$ are $C^{2,\alpha}$, however polar normal coordinates do not always serve as such a chart -  the coefficient $G$ is only $C^\alpha$ in general \cite{hartman}. However, it is differentiable twice with respect to $r$ and satisfies:
\begin{equation}\label{Ginitialcond}
    G(r,\theta)\xrightarrow[]{r\to 0} 0  \qquad \text{and} \qquad  G(r,\theta)/r\xrightarrow[]{r\to 0}1 \qquad \text{ for all } \theta\in S^1,
\end{equation}
as well as the Jabobi equation
\begin{equation}\label{jacobieq}
    \partial^2_rG+KG=0.
\end{equation}

\medskip Differentiating \eqref{jacobieq} with respect to $r$ and setting
$$h := \frac{\partial_rG}{G}$$
Gives the \emph{Riccati equation}
\begin{equation}\label{riccatieq}
    \partial_rh+h^2+K=0
\end{equation}
with  the initial condition
\begin{equation}\label{riccatiinit}
    h=1/r + O(r) \text{ as } r\to0.
\end{equation}
In the case where $K$ is constant, the solution to \eqref{riccatieq},\eqref{riccatiinit} is
\begin{equation*}
    \cot_Kr:=
    \begin{cases}
    \sqrt{-K}\coth(\sqrt{-K}r) & K<0, \ r>0\\
    1/r & K=0, \ r>0\\
    \sqrt{K}\cot(\sqrt{K}r) & K>0, \  0<r<\pi/\sqrt{K}\\
    \end{cases}
\end{equation*}
and the solution to the Jacobi equation \eqref{jacobieq} with the initial condition \eqref{Ginitialcond} is
\begin{equation*}
    \sin_Kr:=
    \begin{cases}
    \sinh(\sqrt{-K}r)/\sqrt{-K} & K<0\\
    r & K=0\\
    \sin(\sqrt{K}r)/\sqrt{K} & K>0\\
    \end{cases}
\end{equation*}
It is also useful to introduce the functions
\begin{equation}\label{Phidef}
    \Phi(x):=
    \begin{cases}
        \sqrt {-x} \coth(\sqrt {-x}) & x<0\\
        1 & x=0\\
        \sqrt x\cot (\sqrt x) & 0<x<\pi^2
    \end{cases}
\end{equation}
and
    \begin{equation*}
    \Psi(x):=
    \begin{cases}
        \sinh(\sqrt {-x})/\sqrt{-x} & x<0\\
        1 & x=0\\
        \sin (\sqrt x)/\sqrt{x} & 0<x<\pi^2
    \end{cases}
\end{equation*}
which are strictly decreasing and analytic. Note that $x\cot_{K}(x)=\Phi(Kx^2)$ and $\sin_{K}(x)/x=\Psi(Kx^2)$.

\medskip The following formula for the geodesic curvature of a unit-speed curve in polar normal coordinates will be useful: 
\begin{lemma}\label{geodesiccurvaturelemma}
    Let $(M,g)$ be a $C^{2,\alpha}$ Riemannian surface and let $(r,\theta)$ be polar normal coordinates centered at $p\in M$. If $\gamma$ is a unit speed curve lying in a normal disc centered at $p$, and the function  $\rho(t):=d_g(\gamma(t),p)$ is $C^2$ and satisfies $|\dot\rho|\ne 1$, then
        \begin{equation}\label{Ddotgamma}
           \nabla_{\dot\gamma}\dot\gamma=-\sqrt{1-\dot\rho^2}\left(\frac{\ddot\rho}{1-\dot\rho^2}-h\right)N,
        \end{equation}
        where $N$ is the \emph{inward-pointing unit normal} to $\gamma$, which is defined to be the unique unit vector field along $\gamma$ such that $\left<\dot\gamma,N\right>=0$ and $\left<\partial_r,N\right><0$.
\end{lemma}

The proof, which is a simple computation in local coordinates, appears in  \cite{Thesis}.

\medskip The function $h$ has some intrisic geometric interpretations, which are worth mentioning, but which shall not be used directly here. If we set $\rho = const$ in \eqref{Ddotgamma}, we see that $h$ is (up to a sign) the geodesic curvature of a unit-speed circle. Moreover, $h$ is the Laplacian of the distance function from the point $p$: if $\Delta$ is the Laplace-Beltrami operator of $g$, then
$$ h = \Delta r .$$

\medskip The Riccati equation \eqref{riccatieq} satisfies the following well-known comparison principle (\cite{petersen2006riemannian}, Proposition 25):
\begin{lemma}\label{riccaticomparison}
    Let $h:(0,R]\to \RR$ be differentiable and satisfy
    \begin{equation*}
    \begin{split}
        h'(x)+h^2(x)+K(x) & =0\\
        h(x)  =1/x+O(x) & \text{ as } x\to0
    \end{split}
    \end{equation*}
    where $K:(0,R]\to\RR$ safisfies $|K|\le H$ for a constant $0<H\le\pi^2/R^2$. Then
    \begin{equation*}
        \cot_{H}x\le h(x) \le \cot_{-H}x
    \end{equation*}
    for all $x\in(0,R]$.
\end{lemma}

\begin{corollary}\label{Gcompcor}
    Let $G$ be a solution of \eqref{Ginitialcond} and  \eqref{jacobieq} on some disc $B_R(0)$, and assume $|K|\le H$ for some constant $H>0$. Then
        \begin{equation*}
        \sin_{H}r\le G(r,\theta)\le\sin_{-H}r \qquad \text{ for all } 0<r\le R \text{ and }\theta\in S^1.
    \end{equation*} In particular, if $R\le\pi/2\sqrt{H}$ then
    \begin{equation*}
        G(r,\theta)/r\sim 1 \qquad \text{ for all } 0<r\le R \text{ and }\theta\in S^1.
    \end{equation*}
\end{corollary}

\begin{proof}
	The first assertion follows Lemma \ref{riccaticomparison} by integration, since $h = (\partial/\partial r)(\log G)$ and $\cot_{K}r = (d/dr)\log \sin_K r$. If $R \le \pi/2\sqrt{H}$ then
		\begin{equation*} \frac12 \le \Psi(HR^2) \le \Psi(Hr^2) = \frac{\sin_Hr}{r} \le \frac G r \le \frac{\sin_{-H}r}{r}  = \Psi(-Hr^2) \le \Psi(-\pi^2/4) \le 2.\end{equation*}
\end{proof}

\medskip \medskip A subset $U$ of a Riemannian manifold is said to be \emph{strongly convex} if every $p,q\in U$ are joined by a unique minimizing geodesic in $M$, which lies entirely in $U$. The following classical theorem (\cite{CE}, Theorem 5.14) asserts that every point has a strongly convex neighbourhood.
\begin{theorem}[Whitehead]\label{whiteheadthm}
    Let $(M,g)$ be a complete Riemannian manifold with sectional curvature $\le H$ and injectivity radius $\ge Z$. For every $R\le\frac12\min\{\pi/\sqrt{H},Z\}$ and every $p\in M$, the disc $B_R(p)$ is strongly convex.
\end{theorem}
\begin{corollary}\label{strongconvcor}
    In $C^{2,\alpha}$ Riemannian surfaces with constant $H$, open discs of radius $\le \pi/(2\sqrt{H})$ are strongly convex.
\end{corollary}
\noindent Combining this with Corollary \ref{Gcompcor}, one obtains
\begin{corollary}\label{metriccompcor}
    Let $(M,g)$ be a $C^{2,\alpha}$ Riemannian surface with constant $H$ and let $p\in M$. For every $v,w\in T_pM$ with $\max\{|v|,|w|\}\le \pi/2\sqrt H$,
    \begin{equation*}
        |v-w|_g\sim d_g(\exp_pv,\exp_pw)
    \end{equation*}
    where $|\cdot|_g:=\sqrt{\left<\cdot,\cdot\right>_g}$ is the Riemannian norm on $T_pM$.
\end{corollary}

\medskip It is useful to have a criterion for strong convexity of a disc without reference to a complete surface containing it.

\begin{lemma}\label{strongconvlma}
	Let $R, H>0$ satisfy $HR^2 \le  \pi^2 / 4$. Let $g = dr^2 + G^2 d\theta^2$ be a $C^{2,\alpha}$ metric with constant $H$ on $D := B_R(0) \subseteq \RR^2$, given in polar normal coordinates centered at the origin. Then $(D,g)$ is strongly convex.
\end{lemma}
\begin{proof}
	Let $R_0 < R$ and let $\gamma : [0,1] \to D$ be a constant - speed curve of whose endpoints lie on $\partial B_{R_0}(0)$ and which lies outside $B_{R_0}(0)$. Write $\gamma$ in polar coordinates as $\gamma = \rho e ^ {i \phi}$.  By the inequality $\sqrt{a^2 + b^2} \ge |a| + b^2 / 2 \sqrt{a^2 + b^2}$ (equality iff $b = 0$), we have
	\begin{align*}
		\Len (\gamma) & = \int_0^1 \sqrt{\dot\rho ^ 2 + G (\rho,\phi) ^ 2  \dot\phi^2 } \ge \int_0^1 \left| G(\rho , \phi) \dot \phi \right| + \frac{\dot\rho^2}{2 \sqrt{\dot\rho ^ 2 + G (\rho,\phi) ^ 2  \dot\phi^2}}\\
		& = \int_0^1 \left| G(\rho , \phi) \dot \phi \right| + \dot\rho^2 / 2\Len(\gamma).
	\end{align*}
	By Lemma \ref{riccaticomparison}, the function $h = \partial_rG/G$ is positive, so $G$ and increasing in $r$ (and positive by Corollary \ref{Gcompcor}), and therefore
	\begin{equation*}
		\int_0^1 \left| G(\rho , \phi) \dot \phi \right| \ge \int_0^1 \left| G(R_0 , \phi) \dot \phi \right| = \Len(\gamma_0),
	\end{equation*}
	where $\gamma_0 : = R_0 e ^ {i \phi}$ is the radial projection of $\gamma$ to the circle $\partial B_{R_0}(0)$. Therefore
	\begin{equation*}
		\Len(\gamma) \ge d(p,q) + \frac {1}{2\Len(\gamma)} \int_0^1\dot\rho^2 \ge d(p,q) + \frac{1}{ 2 \Len(\gamma)}(\max\rho - \min\rho)^2.
	\end{equation*}
	
	From this calculation we conclude that if two points $p,q$ are contained in a disc of radius $R_0 < R$, and $\gamma$ is a curve joining $p$ to $q$ which leaves the disc $B_{R_0+\eps}(0)$, for some $\eps>0$, then the length of $\gamma$ is larger than $d(p,q)$ by a positive constant depending on $\eps$ alone. It follows that any sequence of curves $\gamma_n$ joining $p$ to $q$ whose length tends to $d(p,q)$, remains in a compact subset of $D_R(0)$ and so has a subsequence which converges to a length - minimizing curve, which must be a geodesic (see \cite{BI}).

	It remain to prove the uniqueness of the minimizing geodesic joining two points. Suppose by contradiction that $p,q \in D$ are joined by two different minimizing geodesics. Then, in the region bounded between the two geodesics, there will be a point $q'$ which is a cut point of $p$ of minimal distance to $p$. By the assumption $HR^2 \le \pi^2 / 4$, the metric $(D,g)$ has no conjugate points, so the argument in \cite{CE} implies that there is a smooth geodesic $\gamma: [0,T] \to D$ with $\gamma(0) = \gamma(T) = p$. The $\theta$ coordinate of $\gamma$ must attain a local extremum at a point $t \in (0,T)$, whence $(\theta \circ \gamma)'(t) = 0$. But then, by uniquess of solutions to the geodesic equation, $\gamma$ must be a radial geodesic-  a contradiction.
\end{proof}

\medskip In the proof of Theorem \ref{mainthm}, we will construct a metric $g$ on a disc $D \subseteq \RR^2$, rather than a complete surface, but the following lemma guarantees that this is not an issue.

\begin{lemma}\label{metricextlma}
	Let $R, H>0$ satisfy $HR^2 \le  \pi^2 / 4$. Let $dr^2 + G^2 d\theta^2$ be a $C^{2,\alpha}$ metric with constant $H$ on $D := B_R(0) \subseteq \RR^2$. Then there exists a $C^{2,\alpha}$ surface $(M,g)$ with constant $H$ and a point $p \in M$ such that $g = dr^2 + G^2 d\theta^2$ in polar normal coordinates on a strongly convex disc $B_R(p) \subseteq M$.
\end{lemma}

 The straightforward proof of Lemma \ref{metricextlma} is omitted. The idea is to extend the Gauss curvature function $K : D \to \RR$ to an $\alpha$-H{\"o}lder function which is constant outside some larger disc $D'$, then solve the Jacobi equation \eqref{jacobieq} along lines through the origin, thus obtaining an extension of $g$ to a metric $g'$ on $D'$, and finally glue the disc $(D',g')$ to a surface of constant curvature. This procedure is carried out in detail in \cite{Thesis}, Section 2.6.

\medskip The next lemma, regarding stability of solutions to the Ricatti equation, will play a key role.

\begin{lemma}\label{riccatilemma}
    Let $h_i:(0,R]\to\RR$ \textup{(}$i=1,2$\textup{)} be solutions to the Riccati equation:
    \begin{equation}\label{lemmariccati}
        h_i'+h_i^2+K_i=0
    \end{equation}
    \begin{equation}\label{conditionlemma}
        h_i(x)=1/x+O(x) \ \text{ as } x\to 0 \quad \quad i=1,2
    \end{equation}
    where $K_i:(0,R]\to\RR$ are functions satisfying:
    \begin{enumerate}
        \item $\norm{K_i}_\infty\le H$
        \item $\norm{K_i}_\alpha\le L$
        \item $\norm{K_1-K_2}_\infty\le T$
    \end{enumerate}
    for constants $H,L,R,T>0$ and $0<\alpha\le 1$. Assume $HR^2\le \pi^2/4$. Then, for every $r>0$, the function $f:=h_1-h_2$ has the following properties on the interval $[r,R]$:
    \setlength\columnsep{20pt}
    \begin{multicols}{2}
    \begin{enumerate}[\textup{(}a\textup{)}]
        \item $\norm{f}_\infty\lesssim TR$
        \item $\norm{f'}_\infty\lesssim T$
        \item $\norm{f}_\alpha\lesssim TR^{1-\alpha}$
        \item $\norm{f'}_\alpha\lesssim L+TR^{-\alpha}(1+R/r)$
    \end{enumerate}
    \end{multicols}
\end{lemma}

\begin{proof}

    By \eqref{lemmariccati}, we have:
    \begin{equation}\label{f'lemma}
    \begin{split}
        f'(x) & =h_1'(x)-h_2'(x) = h_2(x)^2-h_1(x)^2+K_2(x)-K_1(x)=\\
        & = (h_2(x)-h_1(x))(h_2(x)+h_1(x))+K_2(x)-K_1(x)=\\
        & = -f(x)\left(h_2(x)+h_1(x)-\frac{2}{x}\right)-\frac{2f(x)}{x}+K_2(x)-K_1(x).
    \end{split}
    \end{equation}
    Set $g(x):=x^2f(x)$ for $0<x\le R$ and $g(0):=0$; then $g$ is $C^1$ on $[0,R]$ by \eqref{conditionlemma}, and \eqref{f'lemma} imply:
    \begin{equation}\label{g'lemma}
    \begin{split}
        g'(x) & =2xf(x)+x^2f'(x)\\
        & =-x^2f(x)\left(h_2(x)+h_1(x)-\frac{2}{x}\right)+x^2(K_2(x)-K_1(x))
    \end{split}
    \end{equation}
    By the Riccati comparison principle (Lemma  \ref{riccaticomparison} above), for all $0<x\le R$,
    \begin{equation}\label{hestimates}
        \cot_{H}x \le h_i(x) \le \cot_{-H}x \qquad i=1,2
    \end{equation}
    whence
    \begin{equation}\label{estimateswithPhi}
        \Phi(Hx^2)=x\cot_{H}x\le xh_i(x)\le x\cot_{-H}x=\Phi(-Hx^2).
    \end{equation}
    By assumption, $HR^2\le\pi^2/4$, and its not hard to show that $|\Phi(t)-1|\le 4|t|/\pi^2$ for $|t|\le \pi^2/4$ (e.g. by convexity of the nonnegative functions $1/x-\cot x$ and $\coth x-1/x$ on the intervals $[0,\pi/2]$ and $[-\pi/2,0]$, respectively). Thus 
    \begin{equation*}
        |xh_i(x)-1|\le 4Hx^2/\pi^2 \qquad x\in[0,R].
    \end{equation*}
    This, together with \eqref{g'lemma} and assumption 3, imply that
    \begin{equation}\label{g'lemma2}
        g'(x)\le 8Hx|g(x)|/\pi^2+Tx^2.
    \end{equation}
    Fix $r\in[0,R]$ and $x\in[r,R]$, and let
    \begin{equation*}
        x_0=\sup\{t\in[0,x] \ \vert \ g(t)=0\}.
    \end{equation*}
    Since $g(0)=0$, necessarily $g(x_0)=0$. Exchanging $h_1$ and $h_2$ if necessary, we may assume that $g$ is positive on $(x_0,x)$. Hence \eqref{g'lemma2} becomes
    \begin{equation}\label{g'bound}
        g'(t)\le 8Htg(t)/\pi^2+Tt^2 \quad \quad t\in[x_0,x].
    \end{equation}
    Setting $\sigma=4/\pi^2$, this last equation can be written as:
    \begin{equation*}
        \frac{d}{dt}\left(e^{-\sigma H t^2}g(t)\right)\le Te^{-\sigma H t^2}t^2 \quad \quad t\in[x_0,x].
    \end{equation*}
    Integrating this we get
    \begin{equation*}
    \begin{split}
        e^{-\sigma H x^2}g(x) & \le T\int_{x_0}^xe^{-\sigma H t^2}t^2dt \\
        & \le Tx\int_{0}^xe^{-\sigma H t^2}tdt\\
        & = \frac{Tx}{2\sigma H }\left(1-e^{-\sigma H x^2}\right)\le\frac{T}{2}x^3
    \end{split}
    \end{equation*}
    where the identity $1-e^a<-a$ was applied in the last passage. It follows that:
    \begin{equation}\label{gbound}
        g(x)\le \frac T2 e^{\sigma H x^2}x^3\le \frac T2e^{\sigma HR^2}x^3\lesssim T x^3
    \end{equation}
    because $HR^2\le\pi^2/4$. Conclusion (a) follows, as
    \begin{equation}\label{firstfestimate}
        f(x)=g(x)/x^2\lesssim Tx \le TR.
    \end{equation}
    Now, by \eqref{g'bound} and \eqref{gbound},
    \begin{equation*}
    \begin{split}
        |f'(x)| & =\left|\frac{g'(x)}{x^2}-\frac{2g(x)}{x^3}\right|\\
        & \le  \frac{2\sigma H |g(x)|}{x}+T+\frac{2|g(x)|}{x^3} \\
        & \lesssim T H x^2+T\le T(HR^2+1)\lesssim T
    \end{split}
    \end{equation*}
    hence (b) holds; it also follows that
    \begin{equation*}
    \begin{split}
        |f(x_1)-f(x_2)| & \lesssim T|x_1-x_2| \\
        &  \lesssim TR^{1-\alpha}|x_1-x_2|^\alpha
    \end{split}
    \end{equation*}
    which proves (c).\\
    To obtain (d), we shall bound the $C^{\alpha}$ seminorm of each term in \eqref{f'lemma} separately. Since $|K_i|\le H$, and $1/x$ is the solution to the Riccati equation with $K\equiv 0$, we can apply the estimates (a) and (c), which we have already proved, to the pairs $h_1,1/x$ and $h_2,1/x$, with $T$ replaced by $H$, and get:
    \begin{equation*}
    \begin{split}
        \norm{h_i-\frac{1}{x}}_\infty & \lesssim HR\\
        \norm{h_i-\frac{1}{x}}_\alpha & \lesssim HR^{1-\alpha}
    \end{split}
    \end{equation*}
    This last estimate, together with (a),(c), and the inequality $\norm{\varphi\psi}_\alpha \le \norm{\varphi}_\alpha \norm{\psi}_\infty + \norm{\varphi}_\infty \norm {\psi}_\alpha$, give
    \begin{equation*}
        \norm{f(x)\left(h_1(x)-\frac{1}{x}+h_2(x)-\frac1x\right)}_\alpha\lesssim TR\cdot HR^{1-\alpha}+TR^{1-\alpha}\cdot HR\lesssim (HR^2)\cdot TR^{-\alpha}\lesssim TR^{-\alpha}.
    \end{equation*}
        As for the second term in \eqref{f'lemma}, we have by (a) and (b) and the fact that $\norm{1/x}_\alpha\lesssim r^{-1-\alpha}$ on $[r,R]$:
    \begin{equation*}
        \norm{f/x}_\alpha\lesssim TR\cdot r^{-1-\alpha}+TR^{1-\alpha}\cdot r^{-1}\lesssim TR^{1-\alpha}r^{-1}.
    \end{equation*}

   Finally, the last two terms have, by assumption, $C^\alpha$ seminorms bounded by $L$. This finishes the proof of (d).
\end{proof}

\begin{corollary}
    Let $H>0$. For any continuous function $K:[0,\pi/(2\sqrt H))\to[-H,H]$ there is a unique solution to the Riccati equation \eqref{riccatieq}, \eqref{riccatiinit} and to the Jacobi equation \eqref{jacobieq}, \eqref{Ginitialcond} on the interval $(0,\pi/(2\sqrt{H}))$.
\end{corollary}

\begin{proof}
	Follow the proof of Lemma \ref{riccatilemma} with $H = 0$.
\end{proof}

\medskip Lemma \ref{riccatilemma} implies the following bounds on the difference between $h$ and its constant-curvature counterpart.

\begin{corollary}\label{fcor}
    Let $H,R>0$, $0<\alpha\le 1$, and suppose that $HR^2\le\pi^2/4$. Let $g$ be a $C^{2,\alpha}$ metric with constant $H$ on $B_R(0)\subseteq\RR^2$ given by $g=dr^2+G^2d\theta^2$, and let $h=\partial_rG/G$.\\
    Suppose that there is some $K_0\in [-H,H]$ such that the curvature $K$ satisfies $$|K(q)-K_0|\le T \text{ for all } q\in B_R(0).$$ Then the function $f:=h-\cot_{K_0}r$ satisfies for all $R_0\in(0,R]$:
    \setlength{\columnsep}{-130pt}
    \begin{multicols}{2}
    \begin{enumerate}[\normalfont(A)]
        \item $\norm{f}_\infty \lesssim  TR_0$ on $B_{R_0}(0)$ 
        \item $\norm{\partial_rf}_\infty\lesssim T$ on $B_{R_0}(0)$
        \item $\norm{f}_\alpha\lesssim LR_0$ on $B_{R_0}(0)$
        \item $\norm{\partial_rf}_\alpha\lesssim L\lambda^{-1}$ on $B_{R_0}(0)\setminus B_{\lambda R_0}(0)$ for all $\lambda<1$
    \end{enumerate}
    \end{multicols}
    \noindent where the H{\"o}lder seminorms are with respect to the Euclidean metric, and $L:=H^{1+\alpha/2}$.
\end{corollary}
\begin{proof}
	This is a straightforward application of Lemma \ref{riccatilemma} and Corollary \ref{metriccompcor}. Details can be found in \cite{Thesis}, Corollary 2.17.
\end{proof}

\medskip Using Corollary \ref{fcor}, we can give a $C^{2,\alpha}$ estimate for the distance function from a point to a unit-speed geodesic on a $C^{2,\alpha}$ surface.

\begin{lemma}\label{geoeqlemma}
	Let $H,R>0$ satisfy $HR^2 \le \pi^2/4$ and let $0<\alpha\le 1$. Let $g = dr^2 + G^2 d\theta^2$ be a $C^{2,\alpha}$ metric with constant $H$ on $B_R(0) \subseteq  \RR^2$, given in polar normal coordinates, and let $h = \partial_rG / G$. Let $\gamma: I \to D$ be unit - speed curve not passing through zero, and denote $\rho(t) = d_g(\gamma(t) , 0), t \in I$. 
	\begin{enumerate}
	\item
	The curve $\gamma$ is a minimizing geodesic if and only if $\rho$ is $C^{2,\alpha}$ and
\begin{equation}\label{geoeq}
	\rho''(t) = h(\gamma(t))(1 - \rho'(t)^2) \qquad t \in I.
\end{equation}
	\item
	If $\gamma$ is a geodesic, then
	\begin{equation}\label{rhoddotesteq}
		\norm{\ddot \rho}_\alpha \lesssim m^{-1-\alpha}( 1 + R / m) \qquad \text{ and } \qquad \norm{\ddot \rho}_\infty \lesssim m^{-1},
	\end{equation}
	where $m = \min\rho > 0$.
\end{enumerate}
\end{lemma}

\begin{proof}
Suppose $\rho\in C^{2,\alpha}$ and \eqref{geoeq} holds. By Lemma \ref{geodesiccurvaturelemma}, the curve $\gamma$ is a geodesic provided that $\dot\rho$ is never $1$ or $-1$. If $\dot\rho\equiv \pm 1$ then $\gamma$ is a geodesic by the definition of normal coordinates. Those are the only possibilities: suppose $|\rho'(t)|<1$ for some $t\in I$ and let $J$ be the maximal relatively open subinterval containing $t$ such that $|\dot\rho|<1$ on $J$. Then by the above $\gamma\vert_{J}$ is a geodesic. If $J\ne I$ then $\rho'(a)=\pm 1$ for some $a\in \partial J$. By Gauss's Lemma, $\left<\dot\gamma(a),\partial_r\right>=\pm 1$ and $\left<\dot\gamma(a),\partial_\theta\right>=0$ (the curve $\gamma$ has no corners because $\rho\in C^2$). It follows by uniqueness of geodesics that $\gamma\vert_{J}$ is a radial geodesic so $\dot\rho\equiv \pm 1$ on $J$, a contradiction. Thus any curve satisfying \eqref{geoeq} is a geodesic, and is minimizing by Corollary \ref{strongconvcor}.

It remains to prove the other implication and the estimates \eqref{rhoddotesteq}. By a standard approximation argument, it suffices to prove that \eqref{geoeq} and \eqref{rhoddotesteq} hold when $g$ is a smooth metric.

Let $\gamma:I\to B_R(p)\setminus B_m(p)$ be a unit-speed minimizing geodesic. We assume that $\gamma$ does not pass through the origin (for radial geodesics \eqref{geoeq} and \eqref{rhoddotesteq} hold trivially). Then $\rho$ is smooth because $g$ is, and it satisfies \eqref{geoeq} because of Lemma \ref{geodesiccurvaturelemma}.\\
By Corollary \ref{fcor} with $K_0=0$, we have that $h=1/r+f$ where $f$  satisfies on $B_R(0)$:
\begin{equation*}
\begin{split}
    \norm{f}_\alpha & \lesssim H^{1+\alpha/2} R= (HR^2)^{1+\alpha/2}R^{-1-\alpha}\lesssim R^{-1-\alpha}\\
    \norm{f}_\infty & \lesssim HR= (HR^2)R^{-1}\lesssim R^{-1}.
\end{split}
\end{equation*}
Since $\norm{1/r}_\alpha\lesssim m^{-1-\alpha}$ on $B_R(0)\setminus B_{m}(0)$ we get that $$\norm{h\circ\gamma}_\alpha\lesssim m^{-1-\alpha} \text{ and }\norm{h\circ\gamma}_\infty\lesssim m^{-1}.$$
From the triangle inequality and the fact that $\gamma$ is a minimizing geodesic, it follows that $|I|\le 2R$. Thus by \eqref{geoeq} and the above estimates,
\begin{equation*}
\begin{split}
    \norm{\ddot\rho}_\infty & \lesssim \norm{h\circ\gamma}_\infty\norm{1-\dot\rho^2}_\infty\lesssim m^{-1} \qquad \text{ and }\\
    \norm{\ddot\rho}_\alpha&\le \norm{h\circ\gamma}_\alpha\norm{1-\dot\rho^2}_\infty+\norm{h\circ\gamma}_\infty\norm{1-\dot\rho^2}_\alpha\lesssim m^{-1-\alpha}+m^{-1}\norm{\ddot\rho}_\infty|I|^{1-\alpha}\\
    & \lesssim m^{-1-\alpha}+m^{-2}R^{1-\alpha}\lesssim m^{-1-\alpha}(1+R/m),
\end{split}
\end{equation*}
and \eqref{rhoddotesteq} is proved.
\end{proof}

\subsection{Whitney's extension theorem}\label{whitneysec}

\medskip Whitney's extension theorem \cite{W34} gives, for every $k \in \NN$, neccesary and sufficient conditions for the extendability of a function $f$ defined on a closed set $S \subseteq \RR$, to a function $F$ defined on all of $\RR$ satisfying $F\vert_{S} \equiv f$ and $f \in C^k(\RR)$. We will use a version of Whitney's extension theorem for $C^{k,\alpha}$ functions, due to Merrien \cite{Mer}, with the notation of A. Brudnyi {\&} Y. Brudnyi (\cite{BB}, Theorem 2.52).

\medskip For a finite set $X \subseteq \RR$ and a function $f: X \to \RR$ we define $f[X]$, the divided difference of $f$ with respect to $X$, by recursion on the size of $X$: 
$$f[\{x_0\}] : = f(x_0),$$

$$f[\{x_0,\dots,x_k\}] : = \frac{f[\{x_0,\dots,x_{k-1}] - f[\{x_1,\dots,x_k\}]}{x_0 - x_k}.$$

\begin{theorem}[Whitney]\label{whitneythm}
	Let $0<\alpha\le 1$. Let $S \subseteq \RR$ be a closed set and let $f: S \to \RR$. Assume that
	$$A : = \sup\left\{ \left|f[X]\right|(\diam X)^{1-\alpha} \mid X \subseteq S, \#X = k+2 \right\} < \infty.$$
	Then $f$ admits an extension $F\in C^{k,\alpha}(\RR)$ satsifying $F \vert_S \equiv f$ and 
	$$\norm{f^{(k)}}_\alpha \lesssim A,$$
	with the implied constant depending on $k$ alone.
\end{theorem}

\begin{corollary}\label{whitneyfiniteness}
        Let $F$ be defined on an interval $I \subseteq \RR$, let $L>0$ and suppose that for every subset $S\subseteq I$ consisting of at most $k+2$ points, there is a function $F_S\in C^{k,\alpha}$ which agrees with $F$ on $S$ and satisfies $\norm{F_S^{(k)}}_\alpha\le L$. Then $F\in C^{k,\alpha}$ and $\norm{F^{(k)}}_\alpha\lesssim L$, with the implied constant depending on $k$ alone.
    \end{corollary}

    \begin{corollary}\label{whitneyextcor}
        Let $T_1,T_2,C>0$ and $0<\alpha\le 1$, and let $I\subseteq\RR$ be an interval of length $|I|\lesssim (T_1/T_2)^{1/\alpha}$. Let $f$ be defined on a closed subset $S\subseteq I$. Suppose for every distinct $x,y,z\in S$,
        \begin{equation*}
            |f(x)-f(y)|\le T_1|x-y|
        \end{equation*}
        and
        \begin{equation*}
            \left|\frac{f(x)-f(y)}{x-y}-\frac{f(y)-f(z)}{y-z}\right|\le T_2\cdot  \diam\{x,y,z\}^\alpha.
        \end{equation*}
        Then $f$ admits an extension $F:I\to \RR$ with $\norm{F'}_\infty\lesssim T_1$ and $\norm{F'}_\alpha\lesssim T_2$. 
    \end{corollary}
    \begin{proof}
        Let $\lambda=(T_1/T_2)^{1/\alpha}$ and define $f_\lambda:\lambda^{-1}S\to\RR$ by $f_\lambda(x):=f(\lambda x)$ (where for a set $A$ and a scalar $t$ we let $tA:=\{ta \ \vert \ a\in A\}$). Our assumptions on $f$ imply that, for $x<y<z\in \lambda^{-1}S$,
        \begin{align}
            |f_\lambda(x)-f_\lambda(y)|&\le T_1\lambda\cdot |x-y| \qquad \text{ and }\label{whitneycorineq1}\\
            \left|\frac{f_\lambda(x)-f_\lambda(y)}{x-y}-\frac{f_\lambda(y)-f_\lambda(z)}{y-z}\right|&\le T_2\lambda^{1+\alpha}\cdot  \diam\{x,y,z\}^\alpha=T_1\lambda\cdot\diam\{x,y,z\}^\alpha.\label{whitneycorineq2}
        \end{align}
        Dividing inequality \eqref{whitneycorineq2} by $z-x$ we obtain
        \begin{equation*}
            |f_\lambda[\{x,y,z\}]|\le T_1\lambda\cdot\diam\{x,y,z\}^{\alpha-1}.
        \end{equation*}
        By Theorem \ref{whitneythm}, $f_\lambda$ admits an extension $F_\lambda:\lambda^{-1}I\to\RR$ which is $C^{1,\alpha}$ with $\norm{F_\lambda'}_\alpha\lesssim T_1\lambda$. \\
        By the mean value theorem and \eqref{whitneycorineq1}, there is a point $w\in \lambda^{-1}I$ satisfying $|F_\lambda'(w)|\lesssim T_1\lambda$.\\
        The function $F:I\to\RR$ defined by $F(x):=F_\lambda(\lambda^{-1}x)$ extends $f$ and satisfies $\norm{F'}_\alpha\lesssim T_1\lambda^{-\alpha}=T_2$ and $|F'(\lambda w)|\lesssim T_1$. Since $|I|\lesssim\lambda$, it follows that $\norm{F'}_\infty\lesssim T_1+T_2\cdot \lambda ^\alpha\lesssim T_1$.
    \end{proof}

\section{Coefficients of $C^{2,\alpha}$ metrics in polar coordinates}\label{coeffsec}

\medskip In this section we prove:

\Gthm*

Note that in this theorem, the coordinates $r,\theta$ denote the usual polar coordinates on $\RR^2$, which coincide with the polar normal coordinates of the metric $g$ given in \eqref{metricinpolarconst}.

 The implication $1 \implies 2$  follows immediately from the facts mentioned in subsection \ref{Riccatifacts}. 

Assume now that condition $2$ holds. We will show that $g = dr^2 + G^2 d\theta^2$ is a $C^{2,\alpha}$ metric with constant $\le 16 H$ on $D$. From Lemma \ref{strongconvlma} it will then follow that $(D,g)$ is strongly convex. 

In order to reveal the $C^{2,\alpha}$ regularity of $g$ we need to pass to isothermal coordinates. Let $$\tD : = B_{2R}(0).$$

\begin{proposition}\label{constructionppn}
    There is a $C^{2}$ function $u:{\tilde D}\to\RR$ such that $({D},g)$ is isometric to $({D},\tilde g)$, where $\tilde g$ is a metric on ${\tD}$ given by
    \begin{equation*}
        \tilde g:=e^{u}dz\overline{dz}=e^{u}(dx^2+dy^2).
    \end{equation*}
    Moreover, $$\norm{u}_\infty\lesssim 1.$$
\end{proposition}
\begin{proof}
	First, we may extend the function $K$ to the disk $\tD$, maintaining the estimates $\norm{K}_\infty \le H$ and $\norm{K}_\alpha \le H^{1 + \alpha/2}$. By solving the Jacobi equation \eqref{jacobieq}, \eqref{Ginitialcond}, we obtain an extension of $G$ to $\tD$ with properties 1 - 3 in the hypothesis of the theorem.

    Let $K_n:\tD \to\RR$ be a sequence of smooth functions satisfying
    \begin{itemize}
        \item $\norm{K_n}_\infty\le H$ and $\norm{K}_\alpha \le H^{1 + \alpha/2}$,
        \item $\norm{K_n-K}_\infty\xrightarrow{n\to\infty}0$ on ${D}$,
        \item $K_n\vert_{B_{1/n}(0)}\equiv K(0)$.
    \end{itemize}  
	Such a sequence can be obtained using mollifiers. Let $G_n:{\tD}\setminus\{0\}\to\RR$ be the solution to the Jacobi equation \eqref{jacobieq}, \eqref{Ginitialcond} with $K$ replaced by $K_n$. Then $G_n$ is smooth, and positive by the assumption $HR^2 \le \pi ^2 / 64$ and Corollary \ref{Gcompcor}. Thus
    $$g_n:=dr^2+G_n^2d\theta^2$$
    is a well-defined smooth metric on $\tD$, whose curvature is $K_n$ and is constant near the origin. For each $n$ we now introduce an isothermal coordinate chart on $(\tD,g_n)$, consisting of a diffeomorphism  $\varphi_n:\tD\to\tD$ and a smooth function $u_n:\tD\to\RR$ such that $\varphi_n(0)=0$, $\varphi(\partial \tilde D) = \partial \tilde D$, and
    \begin{equation*}
        \tilde g_n:=\varphi_n^*g_n=e^{u_n}dz\overline{dz}
    \end{equation*}
    (the asterisk denotes pullback). Such a coordinate transformation always exists (see e.g. \cite{KD}).
    
    \begin{lemma}\label{unlma} The functions $u_n$ satisfy
    \begin{align}
        \sup_{\tD}|u_n|&\lesssim 1\qquad \text{ and }\label{uninfty}\\
        \norm{u_n}_{C^{2,\beta}({D})}& \le C(H,R,\beta) \quad \text{ for all } \beta<\alpha.\label{unc2alpha}
    \end{align}
    \end{lemma}
    \begin{proof}
        The estimate \eqref{uninfty} is an immediate consequence of the following theorem by Eilat \cite{matan}:
        
        \begin{theorem}[Eilat]\label{matanthm}
            Let $H,R>0$ satisfy $HR^2<\pi^2/8$. Let $(M,g)$ be a complete Riemannian surface with injectivity radius $\mathrm{inj}(M)>2R$ and curvature $K$ satisfying $|K|\le H$, and let $p\in M$. Suppose that the metric $g$ has the expression  $e^{u}dz\overline{dz}$ in an isothermal coordinate chart $z:\overline{B_R(p)}\to\overline{B_R(0)}$ satisfying $z(p)=0$ and $z(\partial B_R(p))=\partial B_R(0)$. Then 
            \begin{equation}\label{matanestimate}
                \norm{u}_\infty\lesssim1.
            \end{equation}
        \end{theorem}     
         (the completeness and injectivity radius requirements do not pose a problem: thanks to Lemma \ref{metricextlma}, we may extend $(\tilde D, \tilde g_n)$ to complete surfaces with these properties). To obtain \eqref{unc2alpha}, we recall Liouville's equation:
        \begin{equation}\label{liouville}
            \Delta u_n=-2\cdot (K_n\circ \varphi_n)\cdot e^{u_n}.
        \end{equation}
        and make use of the following elliptic regularity estimates (see \cite{GT}, Theorems 3.9 and 4.8):
        \begin{lemma}
            Let $0<\beta<\alpha$ and $0<R_1<R_2$ and let $\Omega_j=B_{R_j}(0)$.  Let $u\in C^2(\Omega_2)$. Then

            \begin{align}
            \sup_{\Omega_1}|\nabla u|&\lesssim\sup_{\Omega_2}|u|+\sup_{\Omega_2}|\Delta u|\label{poissonestimate1}\\
             \norm{u}_{C^{2,\beta}(\Omega_1)}&\lesssim\sup_{\Omega_2}|u|+\norm{\Delta u}_{C^{0,\beta}(\Omega_2)}\label{poissonestimate2}
            \end{align}
            with the implied constants depending on $R_1,R_2$ and $\beta$.
        \end{lemma}
        The implied constants in the rest of this proof may depend on $\beta, H$ and $R$. From \eqref{liouville}, \eqref{uninfty} and the fact that $|K_n|\le H$, it follows that $|\Delta u_n|\lesssim 1$. Let ${D}'':=B_{1.5 R}(0)$. The estimate \eqref{poissonestimate1} then implies that $\sup_{{D}''}|\nabla  u_n|\lesssim 1$, whence $\norm{u_n}_{C^{0,\beta}({D}'')}\lesssim1$. Now \eqref{liouville}, \eqref{uninfty} and the fact that $\norm{K_n}_\alpha\lesssim H^{1+\alpha/2}$ imply that $\norm{\Delta u_n}_{C^{0,\beta}({D}'')}\lesssim 1$, so \eqref{unc2alpha} follows from \eqref{poissonestimate2}.
    \end{proof}

    By Lemma \ref{unlma} and the Arzel{\`a}-Ascoli theorem, we may assume that $u_n\to u$ uniformly on ${D}$ for some $u\in C^2$. Set
    \begin{equation*}
        \tilde g:=e^{u}dz\overline{dz}.
    \end{equation*}
    We finish the proof of Proposition \ref{constructionppn} by proving that the sequence $\varphi_n$ converges uniformly to an isometry $$\varphi:({D},\tilde  g)\to({D},g).$$
    
    \begin{claim*}
        The distance functions $d_{g_n}$ and $d_{\tilde g_n}$, corresponding to the metrics $g_n$ and $\tilde g_n$, converge uniformly on ${D}\times {D}$ to $d_g$ and $d_{\tilde g}$, respectively.
    \end{claim*}
    \begin{proof}
    Let $\eps>0$ and let $n\in\NN$ be so large that $\sup_{{D}}|K-K_n|<\eps$. By Lemma \ref{riccatilemma}, the functions $h:=\partial_rG/G$ and $h_n:=\partial_rG_n/G$ satisfy:
    \begin{equation}
       \sup_{D}|h-h_n| \lesssim \eps r \qquad \text{ on } B_r(0) \text{ for all } 0<r<2R.
    \end{equation}
    Therefore, for all $\theta\in S^1$ and $0<r<2R$,
    \begin{equation}
        \frac{G(re^{i\theta})}{G_n(re^{i\theta})}=\exp\left(\int_0^r(h-h_n)(se^{i\theta})ds\right)= \exp(O(\eps r^2))=1+O(\eps r^2).
    \end{equation}
    By Corollary \ref{Gcompcor}, if $n$ is large enough then $G_n\sim G \sim r$ whence
    \begin{equation}
        G^2-G_n^2=G_n^2((G/G_n)^2-1)\lesssim \eps r^4.
    \end{equation}
    It follows that for every tangent vector  $v=v^\theta\partial_\theta+v^r\partial_r\in T({D}\setminus\{0\})$,
    \begin{equation}
        |v|_g^2-|v|_{g_n}^2=(v^r)^2+G^2(v^\theta)^2-(v^r)^2-G_n^2(v^\theta)^2\lesssim \eps r^4(v^\theta)^2\le \eps r^2|v|^2,
    \end{equation}
        where $|v|_g,|v|_{g_n}$ and $|v|$ denote the norm of $v$ with respect to $g$, $g_n$ and the Euclidean metric, respectively. Thus 
    \begin{equation*}
        |v|_g-|v|_{g_n}\lesssim \eps r^2|v|.
    \end{equation*}
    For any $\gamma:[0,1]\to{D}$, we then have that 
    \begin{equation}\label{lengminuslengn}
        \Len_g(\gamma)-\Len_{g_n}(\gamma)\le\max_{[0,1]}(|\dot\gamma|_g-|\dot\gamma|_{g_n})\lesssim \eps R^2\cdot \max_{[0,1]}|\dot\gamma|.    
    \end{equation}
    Now, since $G\sim r \sim G_n$,  there is a universal constant $Z>0$ such that for every $p,q\in{D}$,
    \begin{equation*}
    \begin{split}
        d_{g}(p,q) & =\inf\{\Len_{g}(\gamma) \ \vert \ \gamma:[0,1]\to{D}, \gamma(0)=p, \gamma(1)=q, |\dot\gamma|\le Z\} \text{ and } \\
        d_{g_n}(p,q) & =\inf\{\Len_{g_n}(\gamma) \ \vert \ \gamma:[0,1]\to{D}, \gamma(0)=p, \gamma(1)=q, |\dot\gamma|\le Z\}
    \end{split}
    \end{equation*}
    (in words: to measure the distances in $g$ and $g_n$ between $p$ and $q$, it suffices to consider curves whose tangents are no larger than $Z$ in Euclidean norm).

    By \eqref{lengminuslengn}, these two infima differ by at most $CR^2Z\eps$, where $C$ is a universal constant. This proves uniform convergence of $d_{g_n}$ to $d_g$. Convergence of $d_{\tilde g_n}$ to $d_{\tilde g}$ is similar: by uniform convergence of $u_n$ to $u$, we have $|v|_{\tilde g_n}\to|v|_{\tilde g}$ uniformly in $v$, where $v$ is a tangent vector with Euclidean norm bounded by some universal constant; the metrics $\tilde g_n$ and $\tilde g$ are all comparable to the Euclidean metric by a universal factor because of \eqref{uninfty}; and the rest of the argument is analogous. The claim is proved.
    \end{proof}
    The metrics $d_g,d_{g_n},d_{\tilde g}$ and $d_{\tilde g_n}$ are all comparable to the Euclidean metric by a universal factor, and therefore the functions $\varphi_n$, which are isometries from $({D},\tilde g_n)$ to $({D} ,g_n)$, are uniformly bi-Lipschitz with respect to the Euclidean metric. By the Arzel{\`a}-Ascoli theorem, (a subsequence of) $\varphi_n$ converges uniformly to a function $\varphi:{D}\to{D}$, and by the claim, for all $p,q\in{D}$,
    \begin{equation}
        d_g(p,q)=\lim_nd_{g_n}(p,q)=\lim_nd_{\tilde g_n}(\varphi_n(p),\varphi_n(q))=d_{\tilde g}(\varphi(p),\varphi(q)).
    \end{equation}
    This proves that $\varphi:({D},\tilde g)\to({D},g)$ is an isometry. Proposition \ref{constructionppn} is proved. 
\end{proof}
    Proposition \ref{constructionppn} shows that $g$ is a $C^2$ Riemannian metric on $D$. It follows from  \eqref{jacobieq} that the curvature of $g$ is $K$. By $(c)$, we have that  $\norm{K}_\infty\le H$, and $\norm{K}_\alpha\le H^{1+\alpha/2}$ with respect to the Euclidean metric on ${D}$. By Corollary \ref{Gcompcor}, $G\sim r$ on ${D}$, so the Euclidean metric on ${D}$ is equivalent, up to a universal factor, to the metric $d_g$ induced by $g$, whence $\norm{K}_\alpha\lesssim H^{1+\alpha/2}$  with respect to the metric $d_g$. Theorem \ref{Gthm} is proved.

\section{Properties of distance functions on $C^{2,\alpha}$ surfaces}\label{distfuncsec}

In this section, we fix $0<\alpha\le1$ and $H,R>0$ satisfying $HR^2 < \pi^2 / 16$,  and a $C^{2,\alpha}$ metric $g$ with constant $\le H$ on the disc 
$$D = B_R(0) \subseteq \RR^2,$$
 given in polar normal coordinates about 0 by the formula
$$g = dr^2 + G^2 d\theta^2.$$
Again we set
\begin{equation}\label{hdef}h = \frac{\partial_r G}{G}.\end{equation}
We also fix a unit - speed geodesic $\gamma:I \to D$, defined on an open, bounded interval $I \subseteq \RR$, which is non-radial, i.e. not of the form $\theta = \mathrm{const}$. By uniqueness of geodesics with given initial conditions, this implies that $|\rho ' (t)| < 1$ for all $t \in I$. Denote
 $$\rho(t) = d_g(\gamma(t) , 0) = r(\gamma(t)), \qquad t \in I.$$

This section is devoted to proving various properties of the function $\rho$.

\begin{lemma}\label{rhoddotlma}
	We have
	\begin{equation}\label{rhoddotrhooveroneminusdotrhosquared}
		\frac{\rho\ddot\rho}{1- \dot\rho^2} \sim 1.
	\end{equation}
	In particular, $\ddot \rho > 0$.
\end{lemma}

\begin{proof}
    Since $h$ satisfies the Riccati equation \eqref{riccatieq},\eqref{riccatiinit} on radial geodesics, Lemma \ref{riccaticomparison} implies that
    \begin{equation*}
        \Phi(H\rho(t)^2)=\rho(t)\cot_{H}(\rho(t))\le \rho(t)h(\gamma(t))\le \rho(t)\cot_{-H}(\rho(t))=\Phi(-H\rho(t)^2),
    \end{equation*}
    so by \eqref{geoeq} and monotonicity of $\Phi$,
    \begin{equation*}
        1/2\le\Phi(\pi^2/16)\le\rho(t)\rho''(t)/(1-\rho'(t)^2)\le\Phi(-\pi^2/16)\le 2.
    \end{equation*}
\end{proof}

The estimate \eqref{rhoddotrhooveroneminusdotrhosquared} implies that
\begin{equation}\label{logderivest}
	- \frac{d}{dt}\log ( 1 - \rho'(t)^2) = \frac{\rho'(t) \rho''(t)}{1 - \rho'(t)^2} \sim \frac{\rho'(t)}{\rho(t)} = \frac{d}{dt}\log \rho(t).
\end{equation}

This gives us the following relationship between the change in the $r$ and $\theta$ coordinates of $\gamma$:

\begin{lemma}\label{phivarylma}
	Let $\phi(t) = \theta(\gamma(t))$. Then for every subinterval $J \subseteq I$,
	\begin{equation}\label{phivary}
		\max_{t_1,t_2 \in J}\frac{\phi'(t_1)}{\phi'(t_2)} \le C  \left(\max_{t_1,t_2 \in J}\frac{\rho(t_1)}{\rho(t_2)}\right)^{C'},
	\end{equation}
	where $C,C'$ are positive universal constants.
\end{lemma}

\begin{proof}
    Let $$ M = \max_{t_1,t_2 \in J}\frac{\rho(t_1)}{\rho(t_2)}.$$
	Since $\gamma$ is unit speed, $\phi' =\sqrt{1-\dot\rho^2}/G$ (we may assume that $\phi' >0$). Hence by Corollary \ref{Gcompcor}, $$\phi' \sim\sqrt{1-\dot\rho^2}/\rho=:\omega.$$
    It therefore suffices to prove \eqref{phivary} with $\phi'$ replaced by $\omega$. By \eqref{logderivest},
    \begin{equation}\label{dlogomega}
        \left|\frac{d}{dt}\log\omega(t)\right|=\left|\frac{1}{2}\frac{d}{dt}\log(1-\rho'(t)^2)-\frac{d}{dt}\log\rho(t)\right|\lesssim\left|\frac{d}{dt}\log\rho(t)\right| \qquad t\in I.
    \end{equation}
    Let $J \subseteq I$ be an interval and let $t_1 < t_2 \in J$. Since $\rho$ is convex, it attains a unique minimum on $[t_1, t_2]$ at some $t_0 \in J$ (possibly $t_1$ or $t_2$) and is monotone on the intervals $[t_1,t_0]$ and $[t_0,t_2]$. Therefore by \eqref{dlogomega},
    \begin{align*}
        \log\frac{\omega(t_1)}{\omega(t_2)} & \lesssim  \int_{[t_1,t_0]}\left|\frac{d}{ds} \log\rho(s)   \right| ds + \int_{[t_0,t_2]} \left| \frac{d}{ds} \log\rho(s) \right|  ds \\
	 & =  \left|\int_{[t_1,t_0]}\frac{d}{ds} \log\rho(s)  ds \right| + \left| \int_{[t_0,t_2]}\frac{d}{ds} \log\rho(s)  ds \right|\\
	& = \left| \log \frac{\rho(t_1)}{\rho(t_0)} \right| + \left| \log \frac{\rho(t_0)}{\rho(t_2)} \right|
	\lesssim \log M,
\end{align*}
    which implies \eqref{phivary}.
\end{proof}

Define $\kappa : I \to \RR$ by
\begin{equation}\label{kappadef}
	\kappa := \frac{1}{\rho^2} \cdot \Phi^{-1}\left(\frac{\rho\ddot\rho}{1-\dot\rho^2}\right).
\end{equation}

where $\Phi$ is defined in \eqref{Phidef}. As the next lemma asserts, $\kappa(t)$ equals the Gauss curvature at some point between $0$ and $\gamma(t)$.

\medskip Let  $[x,y]$ denote the geodesic segment joining two points $x$ and $y$.

\begin{lemma}\label{kappaivt}
	For every $t\in I$ there exists $q \in [0,\gamma(t)]$ such that $K(q) = \kappa(t).$
\end{lemma}

\begin{proof}
    By \eqref{geoeq}, the formula \eqref{kappadef} is equivalent to
    \begin{equation*}
        \Phi(\kappa(t)\rho(t)^2)=\rho(t)h(\gamma(t)),
    \end{equation*}
    which by the definitions of $\Phi$ and $\cot_{K}$ is equivalent to
    \begin{equation*}
        \cot_{\kappa(t)}\rho(t)=h(\gamma(t)).
    \end{equation*}
    Since $h$ satisfies the Riccati equation \eqref{riccatieq},\eqref{riccatiinit} on the geodesic segment $[0,\gamma(t)]$, it follows from Lemma \ref{riccaticomparison} that $K$ cannot be strictly smaller nor strictly larger than $\kappa(t)$ on the entire segment; the Lemma now follows from continuity of $K$.
\end{proof}

From Lemma \ref{kappaivt} and the assumptions $|K| \le H$, we immediately get that

\begin{equation}\label{kappaest}
	 |\kappa| \le H \qquad \text{ on $I$. }
\end{equation}

Set 
\begin{align}
	t_0 & := \mathrm{argmin}\rho,\\
	K_0 & := \kappa(t_0), \label{K0def}\\
	\phi_0(t) & := \int\limits_{t_0}^t\frac{\sqrt{1 - \rho'(s)^2}}{\sin_{K_0}\rho(s)}ds \qquad t \in I. \label{phi0def}
\end{align}

The motivation behind the definition of $\phi_0$ is that it can by expressed in terms of $\rho$ and its derivatives alone, and it provides a good approximation of the function $\phi = \theta\circ\gamma$, according to the following lemma.

\begin{lemma}\label{phiphi0lma}
	$\displaystyle\sup_I \phi'/\phi_0' = \exp\left(O(HR^2)^{1 + \alpha/2}\right).$
\end{lemma}
\begin{proof}
	Recall that $\phi' = \sqrt{1 - \dot\rho^2}/G$ since $\gamma$ is unit - speed, and $\phi_0'= \sqrt{1-\dot\rho^2}/\sin_{K_0}\rho$. The logarithmic derivatives of $\sin_{K_0}r$ and $G$ in the radial direction are $\cot_{K_0}r$ and $h$, respectively, and both $G$ and $\sin_{K_0}r$ vanish at $r=0$. Thus
	\begin{align*}
	\left| \log(\phi'(t)/\phi_0'(t))\right| & = \left| \log\sin_{K_0}\rho(t) - \log G(\rho(t)e^{i\phi(t)}) \right| \\
		& = \left| \int_0^{\rho(t)} \cot_{K_0}u - h(ue^{i\phi(t)})du \right|\\
		& \lesssim H^{1+\alpha/2} R^{2+\alpha},
	\end{align*}
	where in the last passage we have used Lemma \ref{fcor} with $T = 2H^{1+\alpha/2}R^{\alpha}$, since $K$ differs from $K_0$ by at most $2H^{1+\alpha/2}R^{\alpha}$ on the disc $B_R(0)$.
\end{proof}

\begin{corollary}\label{phiphi0cor}
	$\displaystyle\sup_I \phi /\phi_0 = \exp\left(O(HR^2)^{1+\alpha/2}\right)$.
\end{corollary}

\begin{corollary}\label{phi0varycor}
	For every subinterval $J \subseteq I$,
	\begin{equation*}
		\max_{t_1,t_2 \in J}\frac{\phi_0'(t_1)}{\phi_0'(t_2)} \le C  \left(\max_{t_1,t_2 \in J}\frac{\rho(t_1)}{\rho(t_2)}\right)^{C'}.
	\end{equation*}
	where $C,C' >0$ are universal constants.
\end{corollary}

\begin{proof}
	This follows from Lemma \ref{phivarylma} and Corollary \ref{phiphi0lma} since $HR^2 \lesssim 1$.
\end{proof}

Let
\begin{equation}\label{f0def}
	f_0 = \frac{\ddot \rho}{1 - \dot \rho ^2} - \cot_{K_0}\rho.
\end{equation}

By the geoedsic equation \eqref{geoeq}, we have
\begin{equation}\label{f0isf} f_0 = f \circ \gamma,\end{equation}
where the function $f : D \to \RR$ is defined by
$$ f =  h - \cot_{K_0}r.$$
The Riccati equation \eqref{riccatieq} implies that $f$ satisfies the differential equation
\begin{equation}\label{fKeq} f^2 + 2 f \cot_{K_0}r + \partial_rf = K_0 - K.\end{equation}
The following proposition is the main result of this section. It gives some estimates on $f_0$ which, in some sense, characterize $\rho$ as a distance function from a point to a unit-speed geodesic on  a $C^{2,\alpha}$ surface. To avoid carrying the constant $H$ around we set
$$H = 1.$$

\begin{proposition}\label{f0ppn}
	Let $J \subseteq I$ be an interval, and suppose that there exists $r>0$ such that
\begin{equation}\label{rhosimr}r/2\le\rho\le 2r \qquad \text{on } J.\end{equation}
 Define $f_0$ by \eqref{f0def}. Then for every $t,t_1,t_2,t_3,t_4 \in J$,
\begin{align}
		|f_0(t)|&\lesssim r^{1+\alpha},\label{f0est1}\\
		\left|\frac{\Delta f_0}{\Delta \rho}(t_1,t_2)\right|&\lesssim r^\alpha+ \xi(t_1,t_2)\label{f0est2}\\
    \left|\frac{\Delta f_0}{\Delta \rho}(t_1,t_2) - \frac{\Delta f_0}{\Delta \rho}(t_3,t_4) \right| &  \lesssim \diam\{t_1,t_2,t_3,t_4\}^\alpha + \xi(t_1,t_2) + \xi(t_3,t_4),\label{f0est3}
\end{align}

where $\diam$ is diameter, and we denote

$$\frac{\Delta f_0}{\Delta \rho}(t,t') : = \frac{f_0(t) - f_0(t')}{\rho(t) - \rho(t')}, \qquad \xi(t,t') : = \rho(t)^{1+\alpha}|\phi_0'(t)|^\alpha|t - t'|^\alpha/|\rho(t) - \rho(t')|.$$

If $J,J' \subseteq I$ are two subintervals, and there exist $r,r'$ such that $r/2 \le \rho \le 2r$ on $J$ and $r'/2\le \rho \le 2r'$ on $J'$, then for every $t_1,t_2 \in J$ and $t_3,t_4 \in J'$,

\begin{equation}
\begin{split}\label{f0est4}
	\left|\frac{\Delta f_0}{\Delta \rho}(t_1,t_2) \right. &  \left. +2f_0(t_1)\cot_{K_0}\rho(t_1)-\frac{\Delta f_0}{\Delta \rho}(t_3,t_4)-2f_0(t_3)\cot_{K_0}\rho(t_3)\right| \\ 
	& \lesssim  \diam\{t_1,t_2,t_3,t_4\}^\alpha + \xi(t_1,t_2)+ \xi(t_3,t_4).
\end{split}
\end{equation}

%
%
\end{proposition}

\begin{proof}
	Since $g$ is a $C^{2,\alpha}$ metric with constant $1$, we have $|K| \le 1$ and $\norm{K}_\alpha \le 1$. By Lemma \ref{kappaivt}, the value $K_0$ is attained by $K$ on the disc $B_m(0)$, so on the disc $B_{2r}(0)$, we have $|K - K_0| \lesssim r^\alpha$. Applying Corollary \ref{fcor} with $T \lesssim r^\alpha$, we have
	\begin{multicols}{2}
    \begin{enumerate}
        \item[(a)] $\norm{f}_\infty \lesssim r^{1+\alpha}$ on $B_{2r}(0)$
        \item[(b)] $\norm{\partial_rf}_\infty\lesssim r^\alpha$ on $B_{2r}(0)$
        \item[(c)] $\norm{f}_\alpha\lesssim r$ on $B_{2r}(0)$
        \item[(d)] $\norm{\partial_rf}_\alpha\lesssim 1$ on $B_{2r}(0)\setminus B_{r/2}(0)$,
    \end{enumerate}
    \end{multicols} 
where the distances are measured in the Euclidean metric.

\medskip Below we should make use of the fact that 
$$|t-t'| = d_g(\gamma(t),\gamma(t')) \sim |\gamma(t) - \gamma(t')|, \qquad t,t' \in I.$$

From (a) and \eqref{f0isf}, inequality \eqref{f0est1} follows at once. Recall the notation
	$$\gamma = \rho e ^ {i \phi}.$$
	Write $\rho_i = \rho(t_i)$ and $\phi_i = \phi(t_i)$ for $i=1,2,3,4$. Then by \eqref{f0isf}, \eqref{rhosimr} and the estimates (c) and (d),
\begin{align*}
	\frac{\Delta f_0}{\Delta \rho}(t_1,t_2) &= \frac{f\left(\rho_1e^{i\phi_1}\right) - f\left(\rho_2e^{i\phi_1}\right)}{\rho_1 - \rho_2} + \frac{f\left(\rho_2e^{i\phi_1}\right) - f\left(\rho_2e^{i\phi_2}\right)}{\rho_1 - \rho_2} \\
	& = \partial_rf\left(\rho_1e^{i\phi_1}\right) + O(|\rho_1 - \rho_2|^\alpha) +  O\left(r\left|\rho_2 e^{i\phi_1} - \rho_2e^{i\phi_2}\right|^\alpha/|\rho_1 - \rho_2|\right)\\
	& = \partial_rf(\gamma(t_1)) + O\left(|\rho_1 - \rho_2|^\alpha + r^{1+\alpha}\max_J|\phi'|^\alpha|t_1 - t_2|^\alpha/|\rho_1 - \rho_2|\right).
\end{align*}
	By Lemma \ref{phiphi0lma} and Corollary \ref{phi0varycor}, we have $\max_J\phi' \sim \max_J\phi_0' \sim \phi_0(t_1)$, and therefore
\begin{align}\label{Deltaf0vsdelf}
	\frac{\Delta f_0}{\Delta \rho}(t_1, t_2) = \partial_rf(\gamma(t_1)) +  O\left(|\rho_1 - \rho_2|^\alpha + \xi(t_1,t_2)\right).
\end{align}

Inequality \eqref{f0est2} follows from \eqref{Deltaf0vsdelf}, (b) and \eqref{rhosimr}.

If we replace $t_1,t_2$ in \eqref{Deltaf0vsdelf} by $t_3$ and $t_4$, substract this new inequality from \eqref{Deltaf0vsdelf} and use (d), then we get

\begin{equation*}
	\left|\frac{\Delta f_0}{\Delta \rho}(t_1,t_2) - \frac{\Delta f_0}{\Delta \rho}(t_3,t_4)\right| \lesssim |t_1 - t_3|^\alpha +|\rho_1 - \rho_2|^\alpha + |\rho_3 - \rho_4|^\alpha + \xi(t_1,t_2) + \xi(t_3,t_4).
\end{equation*}

Since $\rho$ is 1-Lipschitz, \eqref{f0est3} follows.

 The inequalities (a) and (c) imply that 
	$$\norm{f^2}_\alpha \lesssim \norm{f}_\infty\norm{f}_\alpha \lesssim r^{2 + \alpha} \lesssim 1$$
 on $B_{2r}(0)$, which together with \eqref{fKeq} and the assumption $\norm{K}_\alpha \le 1$ gives
	\begin{equation}\label{2fcotdelrf}\norm{2 f \cot_{K_0}r + \partial_rf}_\alpha \lesssim 1 \qquad \text{ on } B_{4R}(0).\end{equation}
	Inequality \eqref{f0est4} now follows from \eqref{2fcotdelrf} and \eqref{Deltaf0vsdelf}.
\end{proof}

Up to now we have considered the function $\rho = r\circ \gamma$. We conclude this section with a property of $\phi = \theta \circ\gamma$, namely, that it does not exceed $\pi/2$ by much, which is not surprising. 

\begin{lemma}\label{phi0bound}
	There is a universal constant $C_1>0$ such that if $HR^2 \le C_1$, then $$\max\left\{|\phi(t)|,|\phi_0(t)|\right\} \le 3 \pi / 4 \qquad \text{ for all } t \in I.$$
\end{lemma}
\begin{figure}
            \centering
            \includegraphics[width=.3\textwidth]{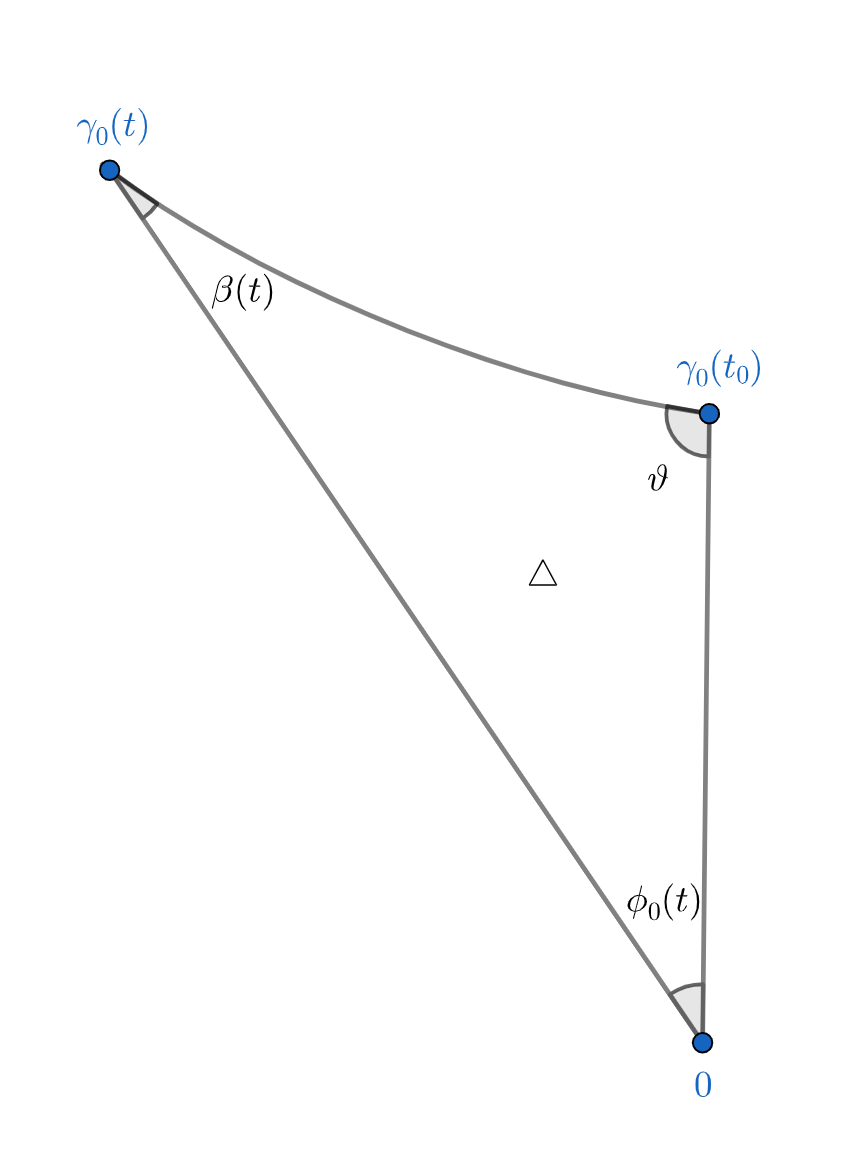}
            \caption{Proof of Lemma \ref{phi0bound}}
            \label{gbfig}
\end{figure}
\begin{proof}
        Fix $t\in I$. For definiteness we assume $t>t_0$, and hence $\phi(t)>0$. The case $t<t_0$ is identical. Consider the triangle $\triangle$ bounded by the three geodesics $[0,\gamma_0(t_0)]$, $\gamma_0\vert_{[t_0,t]}$ and $[\gamma_0(t),0]$ (see Figure \ref{gbfig}). By the Gauss-Bonnet formula,
        \begin{equation}\label{GaussBonnet}
            \phi(t)=\pi-\beta(t)-\vartheta+\int_{\tr}K,
        \end{equation}
    where integration is with respect to the Riemannian area form, and $\vartheta$ and $\beta(t)$ are the interior angles of $\tr$ at $\gamma_0(t_0)$ and $\gamma_0(t)$ respectively.

    Since $\rho$ attains a minimum at $t_0$, the angle $\vartheta$ is at least $\pi/2$ (if $t_0$ is not an endpoint of $I$  then in fact $\vartheta=\pi/2$), and so $\pi - \beta(t) - \vartheta \le \pi / 2$. 

	The Riemannian area form is given in polar coordinates by $G \cdot dr\wedge d\theta$. Using \eqref{jacobieq} and \eqref{Ginitialcond}, Lemma \ref{riccaticomparison} and Corollary \ref{Gcompcor}, we have

\begin{align*}
	\int_{\tr}K & = \int_0^{\phi(t)} \int_0^{\rho(\phi^{-1}(\theta))} K(r,\theta) G(r,\theta) dr d\theta\\
	& \stackrel{\eqref{jacobieq}}{=} - \int_0^{\phi(t)} \int_0^{\rho(\phi^{-1}(\theta))} \partial^2_r G (r,\theta) dr d\theta\\
	& \stackrel{\eqref{Ginitialcond}}{=} \int_0^{\phi(t)} 1 - \partial_rG(\rho(\phi^{-1}(\theta),\theta) d\theta\\
	& \stackrel{\eqref{hdef}}{=} \int_0^{\phi(t)} 1 - (hG)(\rho(\phi^{-1}(\theta),\theta) d\theta\\
	& \stackrel{\ref{riccaticomparison}}{\le} \int_0^{\phi(t)} 1 - \sin_H(\rho(\phi^{-1}(\theta)))\cot_H(\rho(\phi^{-1}(\theta))) d\theta\\
	& =\int_0^{\phi(t)} 1 - \cos(\sqrt{H} \rho(\phi^{-1}(\theta))) d\theta\\
	& \stackrel{\ref{Gcompcor}}{\le} \phi(t)(1- \cos \sqrt{H}R)
\end{align*}
	
	and therefore
	$$ \phi(t) \le \frac{\pi/2}{\cos \sqrt{H}R}.$$
	Thus $\phi(t) \le 3 \pi / 4$ if $HR^2$ is sufficiently small, and by Corollary \ref{phiphi0cor}, the same holds for $\phi_0$.
    \end{proof}

\section{Proof of Theorem \ref{mainthm}}\label{mainthmproofsec}

In this section we prove Theorem \ref{mainthm}. Let $0 < R \le C_1$, where $C_1$ is a universal constant to be determined later. For now we require $C_1 \le \min\{\pi^2/16,\tilde C_1\}$, where $\tilde C_1$ is the constant from Lemma \ref{phi0bound}.  Let $\rho : I \to (0,R]$ satisfy the hypothesis of Theorem \ref{mainthm}:
\begin{center}
\begin{itemize}    \item[$(\star)$] For every subset $S\subseteq I$ consisting of at most 12 points, there exists a $C^{2,\alpha}$ Riemannian surface $(M_S,g_S)$ with constant $\le 1$, a point $p_S \in M_S$ and a unit speed geodesic $\gamma_S : I \to M$ such that $\rho(t) = \rho_S(t) : = d_{g_S}(\gamma_S(t),p_S)$ for all $t \in I$.
\end{itemize}
\end{center}

The following estimates, which were proved in Section \ref{distfuncsec} for distance functions on $C^{2,\alpha}$ surfaces, apply also to the function $\rho$.

\begin{proposition}\label{rhoestppn}
	The function $\rho$ satisfies the following properties:
	\begin{align}
		\frac{\rho\ddot\rho}{1-\dot\rho^2} & \sim 1,\label{rhoest1}\\[10pt]
		|\kappa| & \le H,\label{rhoest2}\\[10pt]
		\max_{t_1,t_2 \in J}\frac{\phi_0'(t_1)}{\phi_0'(t_2)} & \le C  \left(\max_{t_1,t_2 \in J}\frac{\rho(t_1)}{\rho(t_2)}\right)^{C'} \qquad \text{for every subinterval } J \subseteq I,\label{rhoest4}\\[10pt]
		|\phi_0| & \le 3 \pi / 4 \qquad \text{ for all } t \in I.\label{rhoest5}
	\end{align}
	Here $\kappa$ and $\phi_0$ are defined by \eqref{kappadef} and \eqref{phi0def}, respectively, and $C,C'$ are universal constants. \underline{\textit{Moreover, Proposition \ref{f0ppn} holds verbatim for $\rho$}}.
\end{proposition}

\begin{proof}
	For every $S \subseteq I$ of cardinality at most 12, let $\rho_S$ be as in the hypothesis $(\star)$. By Lemma \ref{geoeqlemma}, we have that $\norm{\ddot\rho_S}_\infty$ and $\norm{\ddot\rho_S}_\alpha$ are bounded uniformly over all such $S$. Since the restriction of $\rho$ to any 4-point subset $S$ agrees with $\rho_S$ on $S$, the function $\rho$ itself is also $C^{2,\alpha}$ by Corollary \ref{whitneyfiniteness}. In order to prove the estimates above for the function $\rho$, we can replace the derivatives of $\rho$ by divided differences, apply the results of Section \ref{distfuncsec} to the function $\rho_S$ for a suitable choice of $S$, and conclude the same claim for $\rho$ up to an error which can be made arbitrarily small. 

\medskip For example, consider \eqref{rhoest2}. We have already asserted this inequality in \eqref{kappaest} when $\rho$ was of the form $\rho = d_g(\gamma(\cdot),p)$ for some $C^{2,\alpha}$ surface $(M,g)$ with constant $\le H$. Fix $t \in I$ and $\eps>0$, choose $s_1,s_2\in I $ such that $\diam\{t,s_1,s_2\} < \eps$, and denote $T = \{t,s_1,s_2\}$ and $S = \{s,s'\}$. Since $\rho \in C^{2,\alpha}$, 
$$\kappa(t) = \frac{1}{\rho^2(t)} \Phi^{-1}\left(\frac{\rho(t)\rho[T]}{1 - \rho[S]^2}\right) + O(\eps),$$
where the implied constant does not depend on the choice of $s_i$. The divided difference $\rho[X]$ is defined in \S\ref{whitneysec}. Now use  $(\star)$ and obtain a $C^{2,\alpha}$ surface $(M_T,g_T)$, a point $p_T \in M_T$ and a unit-speed geodesic $\gamma_T$ such that $d_{g_T}(\gamma_T(t),p_T) = \rho(t)$ and $d_{g_T}(\gamma_T(s_i),p_T) = \rho(s_i)$, $i=1,2$. By applying \eqref{kappaest} to the function $\rho_T : = d_{g_T}(\gamma_T(\cdot),p_T)$, we see that
$$\frac{1}{\rho^2(t)} \Phi^{-1}\left(\frac{\rho(t)\rho[T]}{1 - \rho[S]}\right) = \frac{1}{\rho_T^2(t)} \Phi^{-1}\left(\frac{\rho_T(t)\rho_T[T]}{1 - \rho_T[S]^2}\right) \le H + O(\eps),$$
which proves that $\rho$ satisfies \eqref{rhoest1}.

\medskip Each of the estimates mentioned in the Proposition involves the derivatives of $\rho$ up to order two, evaluated at 4 points or less, so 12 points are sufficient to replace all the necessary derivatives by divided differences. Since all of the above estimates were proved in Section \ref{distfuncsec} when $\rho$ is an actual distance function, we may conclude them for the the function $\rho$.
\end{proof}

\medskip A simple consequence of the assumption $(\star)$ and the triangle inequality is that

\begin{equation}\label{rhometriccond} \rho(t) - \rho(t') \le |t - t'| \le \rho(t) + \rho(t') \qquad t,t' \in I. \end{equation}

In particular, $\rho$ is 1 - Lipschitz. We may assume that $\dot \rho \not \equiv\pm1$, because otherwise the conclusion of Theorem \ref{mainthm} is trivial. By \eqref{rhoest1}, we have $\ddot \rho > 0$ which, since $\rho$ is 1- Lipschitz, also implies that $|\dot\rho| < 1$. It also follows that $\rho$ has a unique minimum, which without loss of generality is attained at $0$. Set
$$m := \min\rho = \rho(0).$$
As before, we work in polar coordinates $(r,\theta)$ in $\RR^2$. We shall construct

\begin{enumerate}
	\item  a positive function $G$ on the disk $D:=B_R(0)$, satisfying conditions (a) - (c) of Theorem \ref{Gthm}, and 
	\item a curve $\gamma = \rho e^{i\phi} : I \to D$ satisfying
\begin{equation}\label{rhoconstructionunit} \dot \rho ^2 + (G\circ\gamma) ^2 \dot \phi^2 = 1 \end{equation} and
\begin{equation}\label{rhoconstructiongeoeq}\ddot\rho = (h\circ\gamma)(1 - \dot\rho^2).\end{equation}
\end{enumerate}

It will then follow from Theorem \ref{Gthm} that  the metric
\begin{equation}\label{gdef}g = dr^2 + G^2 d\theta^2\end{equation}
 is a strongly convex $C^{2,\alpha}$ metric on $D$ with constant $\lesssim 1$, and, by Proposition \ref{geoeqlemma}, that  the curve $\gamma$ is a unit - speed minimizing geodesic in $(D,g)$. Since $g$ takes the form \eqref{gdef}, it will follow that
$$d_g(\gamma(t),p) = \rho(t) \qquad t \in I.$$
By Lemma \ref{metricextlma}, the metric $g$ is extendable to a $C^{2,\alpha}$ surface with constant $\lesssim 1$, and so the proof of Theorem \ref{mainthm} will be complete.

Define the functions $\kappa,K_0, \phi_0$ and $f_0$ as in \eqref{kappadef}, \eqref{K0def}, \eqref{phi0def} and \eqref{f0def}. Define a curve $\gamma_0 : I \to D$ by
$$\gamma_0 = \rho e^{i\phi_0}.$$
\begin{proposition}\label{fprop}
	There exists a function $f : D \to \RR$ which is $C^1$ in the  $r$-coordinate, extends $f_0$ from $\gamma_0$ to $D$ in the sense that 
	\begin{equation}\label{fext}
		f \circ \gamma_0 = f_0,
	\end{equation}
	and satisfies
    \begin{align}
		f\vert_{B_{m/2}(0)}&\equiv 0\label{fvanish}\\
		\norm{f/r}_\infty&\lesssim 1,\label{finftycondition}\\
		\norm{f^2+2f\cot_{K_0}r+\partial_rf}_\alpha & \lesssim 1.\label{fholdercondition}
	\end{align}
\end{proposition}

Let us first proceed with the proof of Theorem \ref{mainthm}, assuming Proposition \ref{fprop}. The latter provides us with a function $f$ satisfying $f \circ \gamma_0 = f_0$. Now, the function $G$ should be defined by the relation $f = \partial_rG/G - \cot_{K_0}r$. But we must also bear in mind the $\gamma$ should satisfy \eqref{rhoconstructionunit}. In order to achieve both goals, we first construct $G$ on $\gamma_0$, then adjust the $\theta$-coordinate of $\gamma_0$ to obtain the final curve $\gamma$, and then extend $G$ from $\gamma$ to the disc $D$. Set
\begin{equation*}
	G_0(t) := (\sin_{K_0}\rho(t))\exp \int\limits_0^{\rho(t)}f(ue^{i\phi_0(t)}) du, \qquad t \in I.
\end{equation*}
Then $G_0$ is positive, since $\rho \le R \le \pi / 2 \le \pi / (2 \sqrt{K_0})$ by \eqref{rhoest2}. Let
$$\phi(t) := \int\limits_0^t \frac{\sqrt{ 1 - \rho'(s) ^2 } }{ G_0(s) } ds, \qquad t \in I.$$
and define the curve $\gamma$ by
$$\gamma = \rho e ^ {i\phi}.$$

Thanks to the properties of $f$, the change from $\phi_0$ to $\phi$ is not too dramatic.

\begin{lemma}\label{phibilip}
	For every $\eps>0$ there exists $C > 0$ such that if $R\le C$ then the functions $\phi \circ \phi_0^{-1}$ and $\phi_0 \circ \phi^{-1}$ are both $1 + \eps$  - Lipschitz.
\end{lemma}
\begin{proof}
	By the chain rule and \eqref{finftycondition}, for every $t \in I$,
	\begin{align*}
		(\phi\circ\phi_0^{-1})'(t) = \frac{\phi'(t)}{\phi_0'(t)} = \exp\int_0^{\rho(t)}f\left(ue^{i\phi_0(t)} \right) du = \exp(O(R^2)).
	\end{align*}
\end{proof}
By Lemma \ref{phibilip} and \eqref{rhoest5} , there exists a universal constant $\tilde C_2>0$ such that if $R \le \tilde C_2$ then there is a homeomorphism $F : D \to D$, bi-Lipschitz with universal constants, which preserves the $r$ - coordinate and satisfies $F(re^{i\phi_0(t)}) = re^{i\phi(t)}$ for every $t \in I$; thus
$$F \circ \gamma_0 = \gamma.$$

Define $G : D \to \RR$ by
\begin{equation}\label{Gdef}
	G(r e ^ { i \theta }) : = (\sin_{K_0}r) \exp \int \limits _0 ^r (f \circ F ^ {-1})\left (u e ^ {i \theta } \right) du.
\end{equation}

Then by Proposition \ref{fprop}, $G$ is $C^2$ in the $r$ variable, and satisfies $G \to 0$ and $G / r \to 1$ as $r \to 0$. The logarithmic derivative of $G$ in the $r$ direction is

\begin{equation}\label{heq}
	h : = \frac{\partial_r G}{G} = \cot_{K_0} r + f \circ F ^ {-1} .
\end{equation}

Differentiating with respect to $r$ again and rearranging we have

\begin{equation}\label{GrrG}
	\frac{\partial_r^2 G}{G}  = (f^2 + 2 f \cot_{K_0} r + \partial _r f - K_0) \circ F ^ {-1}. 
\end{equation}

Set 

$$ K = - \frac{\partial _r ^2 G} {G}.$$

By \eqref{GrrG}, \eqref{fholdercondition} and the fact that $F^{-1}$ is Lipschitz with a universal constant, we have 
\begin{equation}\label{Kalphabound}\norm {K} _\alpha \le L_0\end{equation}
 with respect to the Euclidean metric on $D$, where $L_0$ is a universal constant. Moreover, $K(0) = K_0 \le 1$ by \eqref{GrrG} and \eqref{finftycondition}, and so, since $R \le 1$ we also have 
\begin{equation}\label{Kinftybound}\norm{K}_\infty \le 1 + L_0.\end{equation}
Let 
$$\tilde C_3 = \frac{\pi}{16 \cdot \sqrt{1 + L_0}}$$

and  

$$C_1 = \min\{\tilde C_1, \tilde C_2, \tilde C_3\},$$

where $\tilde C_1$ is the constant from Lemma \ref{phi0bound}, and assume that $R \le C_1$. By \eqref{Kalphabound} and \eqref{Kinftybound},
 $$\min\left\{\norm{K}_\infty, \norm{K}_\alpha^{\frac{1}{1 + \alpha / 2}}\right\} \le 1 + L_0 \le \frac{\pi^2}{64R^2}.$$

Thus $G$ satisfies conditions (a) - (c) of Theorem \ref{Gthm}. It remain to verify \eqref{rhoconstructionunit} and \eqref{rhoconstructiongeoeq}.

By \eqref{Gdef} and \eqref{phi0def},
    \begin{equation*}
    \begin{split}
        \rho'(t)^2+G(\gamma(t))^2\phi'(t)^2 & = \rho'(t)^2+\phi'(t)^2(\sin_{K_0}\rho(t))^2\exp\left(2\int_{0}^{\rho(t)}f(ue^{i\psi^{-1}(\phi(t))})du\right)\\
        & = \rho'(t)^2+\phi'(t)^2(\sin_{K_0}\rho(t))^2\exp\left(2\int_{0}^{\rho(t)}f(ue^{i\phi_0(t)})du\right)\\
        & = \rho'(t)^2+\phi'(t)^2G_0(t)^2\\
        & = 1.
    \end{split}
    \end{equation*}
    By \eqref{f0def}, \eqref{fext} and \eqref{heq},
    \begin{equation*}
    \begin{split}
        \rho''(t)
        & = (f_0(t)+\cot_{K_0}(\rho(t)))(1-\rho'(t)^2)\\
        & = (f(\gamma_0(t))+\cot_{K_0}(\rho(t)))(1-\rho'(t)^2)\\
        & = h(F(\gamma_0(t)))(1-\rho'(t)^2)\\
        & = h(\gamma(t))(1-\rho'(t)^2).
    \end{split}
    \end{equation*}

This finishes the proof of Theorem \ref{mainthm}. \qed

\begin{proof}[Proof of Proposition \ref{fprop}]

We shall construct the extension $f$ at each scale of $r$ separately, and then glue the extensions together using a partition of unity. 
For $k \in \ZZ$, let $A_k$ denote the annulus
$$A_k : = \{2^{k-1} < r < 2^{k+1} \},$$
and write
 $$I_k : = \gamma_0^{-1}(A_k).$$ 

Thus  $\rho\sim 2^k$ on $I_k$. By \eqref{rhoest4}, there exists a constant $\lambda_k>0$ such that $$\phi_0' \sim \lambda_k \qquad \text{ on } I_k.$$

By Corollary \ref{Gcompcor}, $$\phi_0'=\sqrt{1-\dot\rho^2}/\sin_{K_0}\rho \sim \sqrt{1-\dot\rho^2}/\rho,$$  

so 

\begin{equation}\label{2klambdaksim}\sqrt{1 - \dot \rho^2} \sim 2^k\lambda_k \qquad \text{ on } I_k.\end{equation}

\begin{lemma}\label{fklemma}
    For each $k\in\ZZ$ such that $I_k \ne \varnothing$, there is a function $f_k:A_k\to\RR$ satisfying $f_k(\gamma_0(t))=f_0(t)$ for all $t\in I_k$, as well as:
    \begin{equation}\label{fkestimates}
    \norm{f_k}_\infty\lesssim 2^{(1+\alpha)k}, \quad  \norm{f_k}_\alpha\lesssim 2^k, \quad  \norm{\partial_rf_k}_\infty\lesssim 2^{\alpha k}, \quad  \text{ and } \quad \norm{\partial_rf_k}_\alpha\lesssim 1.
\end{equation}
\end{lemma}

\begin{proof}

Recall that by \eqref{rhoest4}, \begin{equation}\label{phi0bound2}|\phi_0| \le \theta_0 \quad \text{for some universal constant} \quad \theta_0 < \pi. \end{equation}

Let $k \in \ZZ$. By \eqref{phi0bound2}, it suffices to define $f_k$ on the sector

$$A_k' : = \{ re^{i\phi_0(t)} \mid r \in \left[2^{k-1},2^{k+1}\right], \ t \in I_k \} \subseteq A_k.$$

\begin{figure}
            \centering
            \includegraphics[width=.9\textwidth]{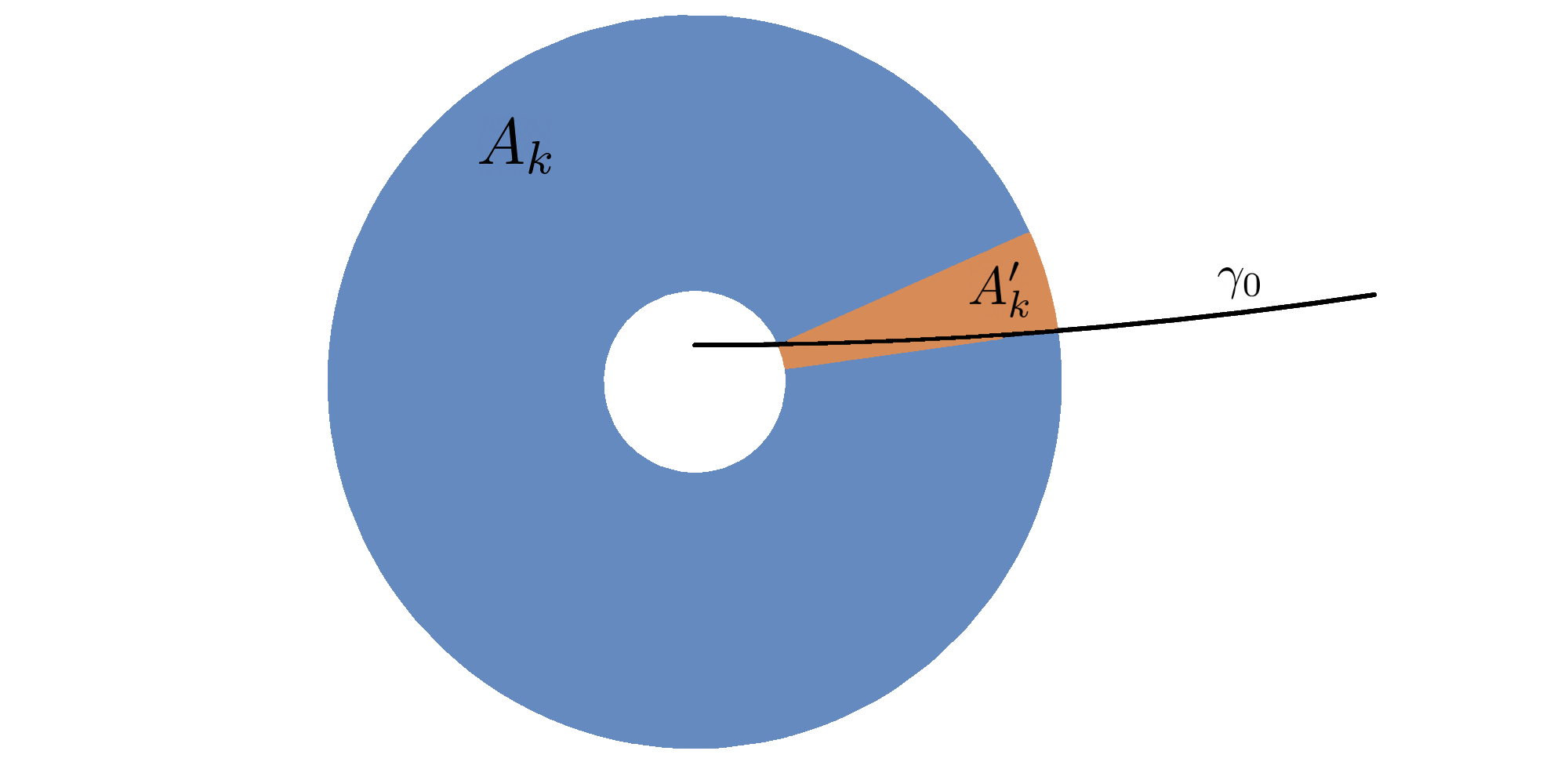}
            \caption{The set $A_k'$.}
            \label{Akprime}
\end{figure}

(see Figure \ref{Akprime}). Indeed, once an extension to $A_k$  has been constructed, it is straightforward to extend it further to $A_k$ maintaining the estimates \eqref{fkestimates}.

Since $I_k \ne \varnothing$, and $\rho$ is convex by \eqref{rhoest1}, $I_k$ consists of either one or two relatively open subintervals of $I$. Assume first that $I_k$ is connected.

To describe the extension $f_k$, we consider several cases.

\begin{itemize}

\item \textbf{ Case I :} $|\rho'(s)| \le 1/2$ for some $s \in I_k$.

\medskip

In this case, \eqref{2klambdaksim} implies that
\begin{equation}\label{2klambdakcase1} \sqrt{1-\rho'(t)^2} \sim 2^k\lambda_k \sim 1 \qquad \text{ for all } t \in I_k. \end{equation}

We apply Proposition \ref{f0ppn} with $J = I_k$ (recall that Proposition \ref{f0ppn} applies to $\rho$ by Proposition \ref{rhoestppn}. It follows from \eqref{f0est2}  that for every $t_1,t_2\in I_k$,
\begin{align*}
    |f_0(t_1)-f_0(t_2)| & \lesssim 2^{k\alpha}|\rho(t_1)-\rho(t_2)|+2^{k(1+\alpha)}\lambda_k^\alpha|t_1-t_2|^\alpha\\
		& \lesssim 2^k|\rho(t_1)-\rho(t_2)|^\alpha+2^k|t_1-t_2|^\alpha\\
		& \lesssim 2^k|t_1-t_2|^\alpha,
\end{align*}
where in the second passage we have used $\rho \sim 2^k$ and $2^k\lambda_k\sim 1$, and in the third passage we have used $|\dot\rho|<1$.
Thus 
	\begin{equation}\label{f0alphacase1}\norm{f_0}_\alpha \lesssim 2^k \quad \text{ on } I_k.\end{equation}
Define $f_k:A_k' \to \RR$ by
\begin{equation}\label{fkdefcase1}
	f_k\left(re^{i\phi_0(t)}\right) := f_0(t), \qquad r \in \left[2^{k-1},2^{k+1}\right], \ t \in I_k.
\end{equation}
Then $\norm{f_k}_\infty \lesssim 2^{(1+\alpha)k}$ by \eqref{f0est1}. The last two inequalities in \eqref{fkestimates} are trivial since $f_k$ does not depend on $r$. It thus remains to prove that $\norm{f_k}_\alpha \lesssim 2^k$ on $A_k'$.

From \eqref{2klambdakcase1} and \eqref{f0alphacase1} and \eqref{holderprops} it follows that
	\begin{equation}\label{fkalphacompu1} \norm{f_k}_\alpha=\norm{f_0\circ\phi_0^{-1}\circ\theta}_\alpha\le \norm{\theta}_1^\alpha\norm{\phi_0^{-1}}_1^\alpha\norm{f_0}_\alpha\lesssim 2^{-\alpha k}\lambda_k^{-\alpha}\norm{f_0}_\alpha \lesssim 2^k.\end{equation}

(here $\norm{\cdot}_1$ is Lipschitz seminorm). Thus $\norm{f_k}_\alpha\lesssim 2^k$ on $A_k'$, which finishes the proof of case I.
\medskip

Denote the length of an interval $J \subseteq \RR$ by $|J|$.

\item \textbf{Case II :} $|I_k| \le \left(2^k\lambda_k^\alpha\right)^{\frac{1}{1-\alpha}}$

In this case we define $f_k$ by \eqref{fkdefcase1} again. It follows from \eqref{f0est2} that for every $t_1,t_2\in I_k$,
\begin{align*}
    |f_0(t_1)-f_0(t_2)| & \lesssim 2^{k\alpha}|\rho(t_1)-\rho(t_2)|+2^{k(1+\alpha)}\lambda_k^\alpha|t_1-t_2|^\alpha\\
    & \le 2^{k\alpha}|t_1 - t_2|+2^{k(1+\alpha)}\lambda_k^\alpha|t_1-t_2|^\alpha.
\end{align*}
The assumption $|I_k| \le \left(2^k\lambda_k^\alpha\right)^{\frac{1}{1-\alpha}}$ implies that the first term is dominated by the second, and therefore $\norm{f_0}_\alpha \lesssim 2^{k(1+\alpha)}\lambda_k^\alpha$, and we can repeat the calculation \eqref{fkalphacompu1}. 

\medskip

\item \textbf{Case III :} $|\rho'(t)| > 1/2$ for all $t \in I_k$, and $|I_k| \ge \lambda_k^\alpha 2^{k(1+\alpha)}$.

\medskip

This case is the most involved. The idea is to first find a $C^{1,\alpha}$ function depending only on $r$, and which agrees with $f_0$ on a discrete net, and then to add to it a correction which is purely a function of $\theta$ so that the resulting function agrees with $f_0$ on all of $\gamma_0\cap A_k$. Our tools will be Proposition \ref{f0ppn} and Whitney's extension theorem in one dimension. Note that we may use Propostion \ref{f0ppn} thanks to Proposition \ref{rhoestppn}. Set
 $$\delta_k:=\lambda_k^\alpha 2^{k(1+\alpha)}$$
 and let $\{s_j\}_{j\in\mathcal J}$ be a $\delta_k$-net in $I_k$. By $\delta_k$-net we mean that $|s_j-s_{j'}|\gtrsim \delta_k$ for every $j,j'\in\mathcal J$ and that for every $t\in I_k$ there is some $j\in\mathcal J$ such that $|s_j-t|\le \delta_k$. Since $I \ge \delta_k$, we may assume that there are at least two distinct $s_j$'s. Let 

\begin{equation}\label{xjdef}x_j:=\rho(s_j).\end{equation}

Note that the $x_j$ are distinct since $\rho$ is monotone; in fact,  since $\dot\rho \sim 1$,
\begin{equation}\label{sjxj}|s_j - s_{j'}| \sim |x_j - x_{j'}| \quad \text{ for all $j \in \mathcal J$} \end{equation}

Define a function $$y:\{x_j\}_{j\in\mathcal J}\to\RR$$ by 
\begin{equation}\label{ydef}y(x_j):=f_0(s_j),\qquad j\in\mathcal{J}.\end{equation}

\begin{claim*} The function $y$ can be extended from the set $\{x_j\}$ to the interval $[2^{k-1}, 2^{k+1}]$ while satisfying 
\begin{equation}\label{yests}\norm{y}_\infty\lesssim 2^{k(1+\alpha)}, \qquad \norm{y'}_\infty\lesssim 2^{k\alpha} \qquad \text{ and } \quad \norm{y'}_\alpha\lesssim 1.\end{equation}
\end{claim*}

\begin{proof}
By \eqref{f0est1}, we have for all $j\in\mathcal J$,
\begin{equation*}
\begin{split}
    |y(x_j)|  \lesssim 2^{k(1+\alpha)}.
\end{split}
\end{equation*}
By \eqref{f0est2}, for $j\ne j'\in\mathcal{J}$,
\begin{equation*}
\begin{split}
    |y(x_j)-y(x_{j'})| & \lesssim 2^{k\alpha}|x_j-x_{j'}|+2^{k(1+\alpha)}\lambda_k^\alpha|s_j-s_{j'}|^\alpha\\
    \text{\scriptsize$\{s_j\}$ is a $\delta_k$-net}\qquad &  \le 2^{k\alpha}|x_j-x_{j'}|+|s_j-s_{j'}|^{1+\alpha}\\
    \text{\scriptsize\eqref{sjxj}}\qquad& \lesssim 2^{k\alpha}|x_j-x_{j'}|+|x_j-x_{j'}|^{1+\alpha}\\
    \text{\scriptsize$x_j,x_{j'}\sim 2^k$}\qquad & \lesssim 2^{k\alpha}|x_j-x_{j'}|.
\end{split}
\end{equation*}
By \eqref{f0est3} and \eqref{sjxj}, for distinct $j,j',j''\in\mathcal{J}$ we have
\begin{equation*}
\begin{split}
    \left|\frac{y(x_j)-y(x_{j'})}{x_j-x_{j'}}-\frac{y(x_{j'})-y(x_{j''})}{x_{j'}-x_{j''}}\right|
    & \lesssim  \diam\{s_j,s_{j'},s_{j''}\}^\alpha\\ &+2^{k(1+\alpha)}\lambda_k^\alpha(|s_j-s_{j'}|^{\alpha-1}+|s_{j'}-s_{j''}|^{\alpha-1})\\
    \text{\scriptsize$\{s_j\}$ is a $\delta_k$-net}\qquad & \lesssim \diam\{s_j,s_{j'},s_{j''}\}^\alpha+2^{k(1+\alpha)}\lambda_k^\alpha\delta_k^{\alpha-1}\\
    & =\diam\{s_j,s_{j'},s_{j''}\}^\alpha+ \delta_k^\alpha\\
    \text{\scriptsize$\{s_j\}$ is a $\delta_k$-net + \eqref{sjxj}}\qquad & \lesssim  \diam\{x_j,x_{j'},x_{j''}\}^\alpha.
\end{split}
\end{equation*}
Corollary \ref{whitneyextcor} now implies that we can extend $y$ from the set $\{x_j\}_{j\in\mathcal J}$ to the  interval $[R_{k-1},R_{k+1}]$ with  $\norm{y'}_\infty\lesssim 2^{k\alpha}$ and $\norm{y'}_\alpha\lesssim 1$. Since $y(x_j)\lesssim 2^{k(1+\alpha)}$ it also follows that $\norm{y}_\infty\lesssim 2^{k(1+\alpha)}$. We have proved \eqref{yests}.
 By \eqref{xjdef} and \eqref{ydef}, the function $y\circ r\circ\gamma_0$  agrees with $f_0$ on the net $\{s_j\}_{j\in\mathcal J}$:
 \begin{equation*}
     (y\circ r)(\gamma_0(s_j))=y(\rho(s_j))=y(x_j)=f_0(s_j) \quad \text{ for all } j\in\mathcal J. 
 \end{equation*}
\end{proof}

We now define $f_k:A_k'\to\RR$ by letting
\begin{equation}\label{fkdefcase3}
    f_k(re^{i\phi_0(t)}):=y(r)+f_0(t)-y(\rho(t)), \qquad r \in \left[2^{k-1},2^{k+1}\right], \ t \in I_k.
\end{equation}
We have 
$$f_k(\gamma_0(t))=y(\rho(t))+f_0(t)-y(\rho(t))=f_0(t) \quad \text{ for all } \quad t\in I_k,$$
 and by \eqref{yests}, $f_k$ is $C^{1,\alpha}$ in $r$ with 
\begin{equation}\label{fkrcase3}\norm{\partial_r f_k}_\infty\lesssim 2^{k\alpha} \qquad \text{ and  } \qquad \norm{\partial_r f_k}_\alpha\lesssim 1.\end{equation}

To see that $\norm{f_k}_\infty\lesssim 2^{k(1+\alpha)}$,  note that each point on $A_k'$ is joined to a point on $\gamma_0$ by a radial line segment of length $\lesssim 2^k$, and $|f_0|\lesssim 2^{k(1+\alpha)}$ on $I_k$ by \eqref{f0est1}, so $$\norm{f_k}_\infty\lesssim \norm{f_0\vert_{I_k}}_\infty+ 2^k\cdot\norm{\partial_rf_k}_\infty\lesssim 2^{k(1+\alpha)}+2^k\cdot 2^{k\alpha}\lesssim 2^{k(1+\alpha)}.$$

It remains to estimate $\norm{f_k}_\alpha$. Fix $r\in \left[2^{k-1},2^{k+1}\right]$; we now prove that 

\begin{equation}\label{fkthetaalpha}\norm{f_k\left(re^{i\phi_0(\cdot)}\right)}_\alpha\lesssim \delta_k \qquad \text{ on } I_k.\end{equation}

Let $t,t'\in I_k$, and assume first that $|t-t'|\le \delta_k$. Take $s_j\ne s_{j'}$ with $\diam\{t,t',s_j,s_{j'}\}\le 2 \delta_k$. Then:
\begin{equation*}
\begin{split}
    |f_k(re^{i\phi_0(t)})-f_k(re^{i\phi_0(t')})|& = |y(r)-f_0(t)-y(\rho(t))-y(r)+f_0(t')+y(\rho(t')|\\
    & = |f_0(t)-f_0(t')-y(\rho(t))+y(\rho(t'))|\\
    & = |\rho(t)-\rho(t')|\left|\frac{f_0(t)-f_0(t')}{\rho(t)-\rho(t')}-\frac{y(\rho(t))-y(\rho(t'))}{\rho(t)-\rho(t')}\right|\\
    \substack{\scriptsize\textstyle\text{ $\norm{y'}_\alpha\lesssim 1$ and} \\  \scriptsize\diam\{\rho(t),\rho(t'),x_j,x_{j'}\}\lesssim\delta_k}\qquad & \lesssim |\rho(t)-\rho(t')|\left(\left|\frac{f_0(t)-f_0(t')}{\rho(t)-\rho(t')}-\frac{y(x_j)-y(x_{j'})}{x_j-x_{j'}}\right|+\delta_k^\alpha\right)\\
    & = |\rho(t)-\rho(t')|\left(\left|\frac{f_0(t)-f_0(t')}{\rho(t)-\rho(t')}-\frac{f_0(s_j)-f_0(s_{j'}))}{\rho(s_j)-\rho(s_{j'})}\right|+\delta_k^\alpha\right)\\
    \text{\scriptsize by \eqref{f0est3} and \eqref{sjxj}} \qquad & \lesssim |\rho(t)-\rho(t')|\left(\delta_k^\alpha+2^{k(1+\alpha)}\lambda_k^\alpha(|t-t'|^{\alpha-1}+|s_j-s_{j'}|^{\alpha-1})\right)\\
    \text{\scriptsize\eqref{sjxj}}\qquad  & \lesssim \delta_k^\alpha|t-t'|+\delta_k|t-t'|^\alpha\\
    \text{\scriptsize$|t-t'|\le\delta_k$}\qquad & \lesssim \delta_k|t-t'|^\alpha.
\end{split}
\end{equation*}

Now assume $|t-t'|\ge \delta_k$, and let $s_j\ne s_{j'}$ satisfy $|t-s_j|,|t'-s_{j'}|\le\delta_k$.
 Since $y(\rho(s_{j}))=f_0(s_j)$ for all $j\in\mathcal J$, 
\begin{equation*}
\begin{split}
    |f_k(re^{i\phi_0(t)})-f_k(re^{i\phi_0(t')})| & = |f_0(t)-y(\rho(t))-f_0(t')+y(\rho(t'))|\\
    & \le |f_0(t)-f_0(s_j)+y(\rho(s_j))-y(\rho(t))|\\
    & +|f_0(t')-f_0(s_{j'})+y(\rho(s_{j'}))-y(\rho(t'))|\\
    & \lesssim \delta_k(|t-s_j|^\alpha+|t'-s_{j'}|^\alpha)\\
    & \lesssim \delta_k^{1+\alpha}\\
    & \lesssim \delta_k|t-t'|^\alpha.
\end{split}
\end{equation*}
where in the third passage we have applied the previous calculation to the pairs $t,s_j$ and $t',s_{j'}$. We have proved \eqref{fkthetaalpha}.

Let $z_1,z_2 \in A_k'$ and write $z_j=r_je^{i\phi_0(t_j)}$ for $r_j\in[2^{k-1},2^{k+1}]$ and $t_j \in I_k$. Then
\begin{align*}
    |f_k(z_1)-f_k(z_2)|& \le |f_k(r_1e^{i\phi_0(t_1)})-f_k(r_1e^{i\phi_0(t_2)})|\\
    &+|f_k(r_1e^{i\phi_0(t_2)})-f_k(r_2e^{i\phi_0(t_2)})|\\
    \text{\scriptsize\eqref{fkrcase3} and \eqref{fkthetaalpha}}\qquad& \lesssim \delta_k|t_1-t_2|^\alpha+2^{k\alpha}|r_1-r_2|\\
    \text{\scriptsize$\dot\phi_0\sim\lambda_k$ on $I_k$} \qquad & \lesssim \delta_k\cdot \lambda_k^{-\alpha}|\phi_0(t_1)-\phi_0(t_2)|^\alpha+2^{k\alpha}\cdot |r_1-r_2|\\
    \text{\scriptsize$\delta_k=2^{k(1+\alpha)}\lambda_k^\alpha$ and $|r_1-r_2|\lesssim 2^k$}\qquad& \lesssim 2^k\cdot (2^{k\alpha}|\phi_0(t_1)-\phi_0(t_2)|^\alpha+|r_1-r_2|^\alpha)\\
    & \sim 2^k|z_1-z_2|^\alpha.
\end{align*}

This finishes the proof of \eqref{fkestimates} in case III.

\item \textbf{Case IV:} $|\rho'(t)| > 1/2$ for all $t \in I_k$, and $\left(2^k\lambda_k^\alpha\right)^{\frac{1}{1-\alpha}} \le |I_k| \le \lambda_k^\alpha2^{k(1+\alpha)}.$

In this case, the extension is in fact the same as in Case III, with the interpolant $y$ taken to be affine-linear. Let $s_1,s_2 \in I_k$ satisfy $|s_1 - s_2| \sim I_k$.  Let 
$$y(x)=f_0(s_1)+(f_0(s_2)-f_0(s_1))(x-\rho(s_1))/(\rho(s_2)-\rho(s_1)), \quad x\in \left[2^{k-1},2^{k+1}\right],$$

and define $f_k$ by \eqref{fkdefcase3}. Then $f_0 = f_k\circ\gamma_0$ on $I_k$, and 
$$\partial_rf_k \equiv \frac{\Delta f_0}{\Delta \rho}(s_1,s_2)$$ 
in the notation of Proposition \ref{f0ppn}, so by \eqref{f0est1} and the fact that $\dot \rho \sim 1$, we have 
\begin{equation}\label{fkrinftycase4}\norm{\partial_rf_k}_\infty \lesssim 2^{k\alpha} + 2^{k(1+\alpha)}\lambda_k^\alpha|s_1-s_2|^{\alpha - 1 }\lesssim 2^{k\alpha},\end{equation}
where the second inequality holds because $|s_1 - s_2| \sim |I_k| \ge \left(2^k\lambda_k^\alpha\right)^{\frac{1}{1-\alpha}}$. We also have trivially $\norm{\partial_rf_k}_\alpha = 0$. Since $f_k \circ \gamma_0 = f$, and every point in $A_k'$ is can be joined to a point on $\gamma_0$ by a radial segment of length $\lesssim 2^k$, \eqref{fkrinftycase4} implies that $\norm{f_k}_\infty \lesssim 2^{k(1+\alpha)}$. The proof that $\norm{f_k}_\alpha \lesssim 2^k$ is as in case III.

\end{itemize}

It remains to deal with the case where $I_k$ consists of two distinct intervals. Recall that 
$$I_k = \gamma_0^{-1}\left(\{2^{k-1} < r < 2^{k+1}\}\right) = \rho^{-1}\left((2^{k-1},2^{k+1})\right).$$
 The function $\rho$ attains a minimum at $0$, and so we can write $I_k = I_k^- \cup I_k^+$ where $I_k^- \subseteq (-\infty,0)$ and $I_k^+ \subseteq (0,\infty)$ are intervals.

We may assume that $\rho(0) \le 2^{k-3}$. Indeed, otherwise, by Corollary \ref{phi0varycor} and \eqref{2klambdaksim} we have $1-\rho'(t)^2 \sim 1 -\rho'(0)^2 = 1$ for $t \in I_k$, and we can proceed as in Case I.

So assume that $\rho(0) \le 2^{k-3}$, whence $\rho(t) \ge 4 \rho(0)$ for all $t\in I_k$. We claim that \begin{equation}\label{phi0lower}|\phi_0| \ge \theta_1 \quad \text{ on } I_k \quad \text{for a universal constant $\theta_1 > 0$.}\end{equation}

We prove \eqref{phi0lower} on $I_k^+$, and the proof is identical for $I_k^-$. There exists a unique $t_1 > 0$ satisfying $\rho(t_1) = 2\rho(0)$. Note that $t_1 < \inf I_k^+$. By  \eqref{rhoest1}, we have that  $1-\dot\rho^2$ decreases by at most some constant factor on the interval $[0,t_1]$ (see the computation \eqref{logderivest}). Since $\rho'(0) = 0$, It follows that $\dot\rho\le 1-c$ on $[0,t_1]$, whence $$d\log\rho/dt\le(1-c)/\rho\le (1-c)/\rho(0) \qquad \text{ on } [0,t_1],$$
 for some universal $0<c<1$. Thus 
    \begin{equation*}
        \log 2=\int_0^{t_1}(\log\rho)'(s)ds\le t_1(1-c)/\rho(0).
    \end{equation*}
    Since $\phi_0'(0)=\sqrt{1-\rho'(0)^2}/\sin_{K_0}m=1/\sin_{K_0}m\sim 1/m$ by Corollary \ref{Gcompcor}, where $m = \min\rho=\rho(0)$, the estimate \eqref{rhoest4} implies that $\dot\phi_0\sim 1/m$ on $[0,t_1]$. Since $\phi_0$ is increasing, it follows that for every $t \in I_k^+$,
    \begin{equation*}
        \phi_0(t) \ge \phi_0(t_1)=\int_0^{t_1}\phi_0'(s)ds\sim t_1/m\ge \log 2/(1-c).
    \end{equation*}
Thus \eqref{phi0lower} is proved. Write $A_k' = A_k'^+ \cup A_k'^-$, where
$$A_k'^\pm : = \left\{ re^{i\phi_0(t)} \mid r \in [2^{k-1},2^{k+1}], \ t \in I_k^\pm \right\} \subseteq A_k'.$$
Then \eqref{phi0bound2} and \eqref{phi0lower} imply that the distance between the sets $A_k'^+$ and $A_k'^-$ is $\gtrsim 2^k$. Therefore, we can construct the extension $f_k$ on $I_k^+$ and $I_k^-$ separately, according to the cases above, and it is not hard to see that since the function $f_k$ satisfies \eqref{fkestimates} on $A_k'^+$ and $A_k'^-$, it will satisfy \eqref{fkestimates} on their union $A_k'$. This finishes the proof in the case where $I_k$ consists of two intervals, and thus Lemma \ref{fklemma} is proved.
\end{proof}

We now glue together the local extensions $f_k$ obtained via Lemma \ref{fklemma}. We also define
\begin{equation}\label{fknothing} f_k \equiv 0 \qquad \text{ whenever } I_k = \varnothing.\end{equation}
Let $\varphi_k : \RR^2 \to [0,1]$, $k\in \ZZ$ be smooth functions satisfying
\begin{equation}\label{varphik} \mathrm{supp}\varphi_k \subseteq A_k, \qquad \sum_k\varphi_k \equiv 1, \qquad |\nabla\varphi|\lesssim 2^{-k}, \qquad |\mathrm{Hess}\varphi|\lesssim 2^{-2k}.\end{equation}
(take $\varphi_k = \varphi(2^{-k}r)$, where $\varphi:(0,\infty) \to [0,1]$ is a suitable bump function). Note that each point of $\RR^2$ lies in the support of at most two such functions. Define 
\begin{equation}\label{fconstdef}
	f : = \sum_k\varphi_kf_k.
\end{equation}

Since $\sum_{k}\varphi_k\equiv 1$, and $f_k\circ\gamma_0=f_0$ on each $I_k$, we have that ${}f\circ\gamma_0=f_0$.

The function ${}f$ is $C^1$ in $r$, and
\begin{equation}\label{delrf}
    \partial_rf=\sum_{k}(\partial_r\varphi_k) f_k+\varphi_k(\partial_rf_k).
\end{equation}
By \eqref{varphik} and \eqref{fkestimates}, and the local finiteness of the sums in \eqref{fconstdef} and \eqref{delrf},
\begin{equation}\label{fdelrfbound}
    |{}f|\lesssim r^{1+\alpha} \quad \text{ and } \quad |\partial_r{}f|\lesssim r^\alpha. 
\end{equation}
In particular, condition \eqref{finftycondition} holds, while \eqref{fvanish} follows from \eqref{fknothing}.

To finish the proof it remains to prove \eqref{fholdercondition}, i.e. that $\norm{\Gamma}_\alpha\lesssim 1$, where $$\Gamma:={}f^2+2{}f\cot_{K_0}r+\partial_r{}f.$$
It follows from \eqref{holderprops}, \eqref{varphik}, \eqref{fkestimates}, \eqref{fconstdef} and \eqref{delrf}, that
\begin{align}\label{flocalalpha}
    \norm{{}f\vert_{A_k}}_\alpha&\lesssim 2^k&\text{ and}&& \norm{\partial_r{}f\vert_{A_k}}_\alpha&\lesssim 1 & \text{ for all } k \in \ZZ.
\end{align}
Putting together \eqref{fdelrfbound},\eqref{flocalalpha}, the estimates $$\norm{\cot_{K_0}r\vert_{A_k}}_\infty\lesssim 2^{-k} \quad \text{ and } \quad  \norm{\cot_{K_0}r\vert_{A_k}}_\alpha\lesssim 2^{-k(1+\alpha)},$$ and the fact that $f$ is supported on a disc with radius $\lesssim 1$ , we get
\begin{align}\label{Gammalocal}
    \norm{\Gamma\vert_{A_k}}_\alpha  \lesssim 1 \qquad \text { for all } k \in \ZZ.
\end{align}
Our goal is to make this estimate global. We also note that by \eqref{fdelrfbound},
\begin{equation}\label{Gammabound} \norm{\Gamma\vert_{A_k}}_\infty \lesssim 2^{\alpha k}. \end{equation}

Given any $z,z'\in\RR^2$, we would like to show that 
\begin{equation}|\Gamma(z)-\Gamma(z')|\lesssim |z-z'|^\alpha.\end{equation}
Let $k,k'$ satisfy $z\in A_k$ and $z' \in A_{k'}$, where without loss of generality $k'\ge k$. By \eqref{Gammalocal}, we may even assume that $k'  > k + 2$. This implies that $|z-z'| \gtrsim 2^{k'}$, and therefore it suffices to prove that
\begin{equation}\label{GammazGammaz'}|\Gamma(z)-\Gamma(z')|\lesssim 2^{k'\alpha}. \end{equation}
By \eqref{Gammalocal}, we may also replace $z$ and $z'$ by any other pair of points lying in $A_k,A_{k'}$ respectively. 

If $A_k\cap\gamma_0 = \varnothing$, then by \eqref{fknothing}, we may assume that $f_k\equiv 0 $ in a neighbourhood of $z$, and then by \eqref{Gammabound},

$$|\Gamma(z) - \Gamma(z')| = |\Gamma(z')| \lesssim 2^{k'\alpha}.$$

We can argue similarly if $A_{k'}\cap \gamma_0 = \varnothing$. So from here on we can assume that both $A_k\cap\gamma_0$ and $A_{k'}\cap\gamma_0$ are nonempty, and that $z,z'$ both lie on $\gamma_0$. Write

$$z = \gamma_0(t), \qquad z' = \gamma_0(t').$$

There is no loss of generality in assuming that $0\le t<t'$, because if $t'<0<t$ we can compare both $\Gamma(z)$ and $\Gamma(z')$ to $\Gamma(\gamma_0(0))$. The rest of the cases follow by symmetry. So, to summerize, we must prove that

$$|\Gamma(z) - \Gamma(z')| \lesssim 2^{k'}, \qquad \text{ where } z = \gamma_0(t) \in A_k, \  z' = \gamma_0(t') \in A_{k'} \ , \qquad 0\le t' < t.$$

To this end we would like to use \eqref{f0est4}. Since $\rho(t')=|z'|\ge 2^{k'-1}\ge 2^{k+1}\ge 2|z|=2\rho(t)$, there are $t<s\le s'<t'$ such that $\rho(s)=2\rho(t)$ and $\rho(s')=\rho(t')/2$. It follows that
\begin{equation}\label{rhotminusrhos}
    |\rho(t)-\rho(s)|\sim \rho(t)\sim 2^k \quad \text{ and } \quad  |\rho(t')-\rho(s')|\sim \rho(t')\sim 2^{k'}.
\end{equation}
Using \eqref{rhotminusrhos}, \eqref{fext} and the local estimates \eqref{flocalalpha}, we have
\begin{align*}
    \frac{f_0(t)-f_0(s)}{\rho(t)-\rho(s)}& = \frac{{}f(z)-{}f(\gamma_0(s))}{\rho(t)-\rho(s)}\\
    & = \frac{{}f(\rho(t)e^{i\phi_0(t)})-{}f(\rho(t)e^{i\phi_0(s)})}{\rho(t)-\rho(s)}+\frac{{}f(\rho(t)e^{i\phi_0(s)})-{}f(\rho(s)e^{i\phi_0(s)})}{\rho(t)-\rho(s)}\\
    & =\partial_r{}f(z)+O(2^{k\alpha})
\end{align*}
and similarly 
\begin{equation*}
    \frac{f_0(t')-f_0(s')}{\rho(t')-\rho(s')}=\partial_r{}f(z')+O(2^{k'\alpha}).
\end{equation*}
Thus by \eqref{f0est4},
\begin{equation*}\label{delrfplus2fcotfinal}
\begin{split}
    &|\partial_r{}f(z)+2{}f(z)\cot_{K_0}(|z|)-\partial_r{}f(z')-2{}f(z')\cot_{K_0}(|z'|)|\\
    & = \left|\frac{f_0(t)-f_0(s)}{\rho(t)-\rho(s)}+2 f_0(t)\cot_{K_0}\rho(t)-\frac{f_0(t')-f_0(s')}{\rho(t')-\rho(s')}-2f_0(t')\cot_{K_0}\rho(t')\right|+O(2^{k'\alpha})\\
    & \lesssim |t'-t|^\alpha+2^{k(1+\alpha)}\lambda_k^\alpha|t-s|^\alpha/|\rho(t)-\rho(s)|+2^{k'(1+\alpha)}\lambda_{k'}^\alpha|t'-s'|^\alpha/|\rho(t')-\rho(s')|+2^{k'\alpha}\\
    & \lesssim 2^{k'\alpha}+2^{k(1+\alpha)}\lambda_k^\alpha\cdot 2^{k\alpha}/2^k+2^{k'(1+\alpha)}\lambda_{k'}^\alpha\cdot 2^{k'\alpha}/2^{k'}+2^{k'\alpha}\\
    & \lesssim 2^{k'\alpha}
\end{split}
\end{equation*}
(we have used \eqref{rhometriccond} and \eqref{rhotminusrhos} in the third passage, and \eqref{2klambdaksim} in the penultimate one)

By \eqref{fdelrfbound} and \eqref{flocalalpha}, we have that $\norm{{}f^2}_\alpha\lesssim 2^{k(1+\alpha)}\cdot 2^k\le R^{2+\alpha} \lesssim 1$ on $A_k$ and $|\partial_r({}f^2)|\lesssim 2^{k'(1+\alpha)}\cdot 2^{k'\alpha}\lesssim R^{1+2\alpha}\lesssim 1$ on $B_{2^{k'}}(0)$, so 
$$|{}f(z)^2-{}f(z')^2|=|{}f(\tilde z)^2-{}f(z')^2|+O(2^{k\alpha}) \lesssim 2^{k'\alpha},$$
 where $\tilde z$ is any point in $A_k$ lying on the same radial ray as $z'$. Putting together the last two calculations, we obtain \eqref{GammazGammaz'} and finish the proof of the proposition.
\end{proof}

\end{document}